\documentclass[a4paper,10pt]{article}

\usepackage{amsmath,amsfonts,amsthm}
  
\usepackage{paralist}
\usepackage{graphics} 
\usepackage{graphicx}
  
\usepackage{color}  

\usepackage{geometry}
\usepackage{setspace}
\usepackage{epsfig}
\usepackage[colorlinks=true]{hyperref}
  
\hypersetup{urlcolor=blue, citecolor=red}

\newtheorem{theorem}{Theorem}[section]

\newtheorem{lemma}[theorem]{Lemma}

\newtheorem{remark}{Remark}

\renewenvironment{proof}[1][Proof]{\begin{trivlist}
\item[\hskip \labelsep {\bfseries #1}]}{\end{trivlist}}

\newcommand{\dd}{d}
\newcommand{\y}{\mathbf{y}}
\newcommand{\z}{\mathbf{z}}
\newcommand{\x}{\mathbf{x}}
\newcommand{\q}{\mathbf{q}}
\newcommand{\uu}{\mathbf{u}}
\newcommand{\vv}{\mathbf{v}}
\newcommand{\gu}{\nabla_{\y}\mathbf{u}}

\newcommand{\RR}{\mathbb{R}}
\newcommand{\Rd}{\mathbb{R}^3}
\newcommand{\Sd}{\mathbb{S}^{2}}

\newcommand{\La}{L^2_\alpha}

\newcommand{\OQ}{\Omega\times Q}

\newcommand{\psit}{\widetilde{\psi}}
\newcommand{\phit}{\widetilde{\phi}}

\newcommand{\psih}{\widehat{\psi}}
\newcommand{\phih}{\widehat{\phi}}

\newcommand{\deriv}[1]{\dfrac{\partial}{\partial #1}}

\newcommand{\gi}{\underline{g}}
\newcommand{\gs}{\overline{g}}

\newcommand{\drift}{\rm{P}_{\eta^\perp}\left(\nabla_{\y} \uu\ \eta \right)}

\title{\scshape \large Fragmentation and monomer lengthening of rod-like polymers, a relevant model for prion proliferation}

\author{\scshape \normalsize Ionel Sorin Ciuperca$^1$, Erwan Hingant$^1$,\\ \scshape \normalsize Liviu Iulian Palade$^2$ and Laurent Pujo-Menjouet$^1$\thanks{ciuperca@math.univ-lyon1.fr, hingant@math.univ-lyon1.fr, liviu-iulian.palade@insa-lyon.fr, pujo@math.univ-lyon1.fr}
}

\date{\normalsize September 29, 2011}

\begin{document}

\maketitle

{\footnotesize  {\itshape
 \centerline{$^1$Universit\'e de Lyon, CNRS UMR 5208, Universit\'e Lyon 1, Institut Camille Jordan}
   \centerline{43, blvd. du 11 novembre 1918, F-69622 Villeurbanne cedex, France}
}}
\medskip
{\footnotesize  {\itshape
 \centerline{$^2$ Universit\'e de Lyon, CNRS UMR 5208, INSA de Lyon, Institut Camille Jordan}
   \centerline{21, avenue Jean Capelle, F-69621 Villeurbanne cedex, France}
}}

\medskip

\begin{center} \textbf{Abstract} \end{center}
\noindent The Greer, Pujo-Menjouet and Webb model [Greer {\it et al.},  J. Theoret. Biol., {\bf 242} (2006), 598--606] for prion dynamics was found to be in good agreement with experimental observations under no-flow conditions.  The objective of this work is to generalize the problem to the framework of general polymerization-fragmentation under flow motion, motivated by the fact that laboratory work often involves prion dynamics under flow conditions in order to observe faster processes. Moreover, understanding and modelling the microstructure influence of macroscopically monitored non-Newtonian behaviour is crucial for sensor design, with the goal to provide practical information about ongoing molecular evolution. This paper's results can then  be considered as one step in the mathematical understanding of such models, namely the proof of positivity and  existence of solutions in suitable functional spaces.
To that purpose, we introduce a new model based on the rigid-rod polymer theory to account for the polymer dynamics under flow conditions. As expected, when applied to the prion problem, in the absence of motion it reduces to that in Greer \emph{et al.} (2006). At the heart of any polymer kinetical theory there is a configurational probability diffusion partial differential equation (PDE) of Fokker-Planck-Smoluchowski type.  The main mathematical result of this paper is the proof of  existence of positive solutions to the aforementioned PDE for a class of flows of practical interest, taking into account the flow induced splitting/lengthening of polymers in general, and prions in particular.\medskip \\
{\footnotesize
}


%
%
%
\section{Introduction}
\subsection{Taking space into account for our problem: what is new in biology, what is new in mathematics?}

In 1999, Masel \textit{et al.} \cite{Masel1999} introduced a new model of polymerization in order to quantify
some kinetic parameters of prion replication. This work was based on a deterministic discrete
model developed into an infinite system of ordinary differential equations,
one for each possible fibril length. In 2006, Greer \textit{et al.} in \cite{Greer2006} 
modified this model to create a continuum of possible fibril lengths described 
by a partial differential equation coupled with an ordinary differential equation.
This approach appeared to be ``conceptually more accessible and mathematically more tractable
with only six parameters, each of which having a biological interpretation'' \cite{Greer2006}.
However, based on discussions with biologists, it appeared that these models were not 
well adapted for \textit{in vitro} experiments. In these experiments, proteins
are put in tubes and shaken permanently throughout the experiment
to induce an artificial splitting in order to accelerate the polymerization-fragmentation mechanism.
To the best of our knowledge, dependence of polymer and monomer interaction on the shaking
 orientation and strength, space competition and fluid viscosity had never been taken into account until now.
Thus, it seemed natural to propose a model generalizing the Greer model and adapt it to the 
specific expectations of the biologists. 

We therefore introduce a new model of 
polymer and monomer interacting in a fluid, with the whole system subjected to  motion. 
A large range of \emph{in vitro} experiments involving this protein refers to this protocol
in order to accelerate the polymerization-fragmentation process. Moreover, 
even as our model could be well adapted to other polymer-monomer interaction
studies, we give here a specific application to prion dynamics to make an interesting link
with the previous Masel \emph{et al.} \cite{Masel1999} and Greer \emph{et al.} \cite{Greer2006} models.
On the other hand, due to the complexity of the model, any mathematical
analysis becomes a challenge. We adapt here a technique of semi-discretization in 
time for proving  the main result of existence of positive solutions, we also provide the basis for the numerical approximation of the problem. The mathematical 
novelty of this paper resides in the choice of the \emph{ad hoc} function spaces and the appropiate modification of the existing 
techniques to this new type of problem. Also this work presents an alternative way for proving the existence of positive solutions as compared to the one given by Engler \emph{et al.} in \cite{Engler2006}, Lauren\c{c}ot and Walker in \cite{Laurencot2007} and Simonett and Walker in \cite{Simonett2006}.
It is then useful to those who consider which techniques to use when  proving the existence of positive solutions of this class of equations. 

The objective of this paper is twofold: not only to make a step forward in mathematical modelling of a class of polymer-monomer interaction models, but also to propose,  within a new framework, how to adapt an existing mathematical technique
that will prove the existence of positive solutions to  the problem. 
The biological implications (\emph{e.g.} quantitative and qualitative comparison
with experimental data) of this paper model will be addressed in a subsequent work.
\subsection{The polymer-monomer interaction model: an application to prion dynamics}

Prion proliferation is challenging at both the biological and mathematical levels. Prions are responsible for several diseases such as \emph{bovine spongiform encephalopathy}, \emph{Creutzfeld-Jacob disease}, \emph{Kuru} and it is now commonly accepted that prions are proteins \cite{Prusiner1998}.

For the sake of clarity, we present several fundamental morphological features of prions with relevance to the mathematical modelling of this paper (\textit{i.e.} molecular dynamics of a low enough concentration prion solution).

There are two types of prions: the \emph{Prion Protein Cellular} also called $PrP^C$ and \emph{Prion Protein Scrapie} denoted by $PrP^{Sc}$. It has been proven that $PrP^C$ proteins are naturally synthesized by mammalian cells and consist only of monomers. On the other hand, the infectious  $PrP^{Sc}$ proteins are present only in  pathologically altered cells and exist only in ``polymer''-shape.  The conversion process of a non-pathological into a pathologically modified one consists in attaching the former to an already existing polymer (for details see \emph{e.g.} \cite{Lansbury1995}). As a consequence, the polymers lengthen.  However the sized-up new polymers are fragile, and shorten down their size by splitting whenever the polymer solution is subjected to some flow conditions.  The size lengthening/shortening process takes place continuously, its kinetics being dependent on monomer concentration, flow intensity, polymer size, etc.

Polymers may be seen as string-like molecules \cite{Scheibel2001}.  When polymer proliferation occurs, they do interact to form fibrils;  these latter exhibit  a (physically speaking) more stable structure and appear as rod-like molecules (see figure \ref{fig1}).
\begin{figure}
\begin{center}
\includegraphics[width=4in]{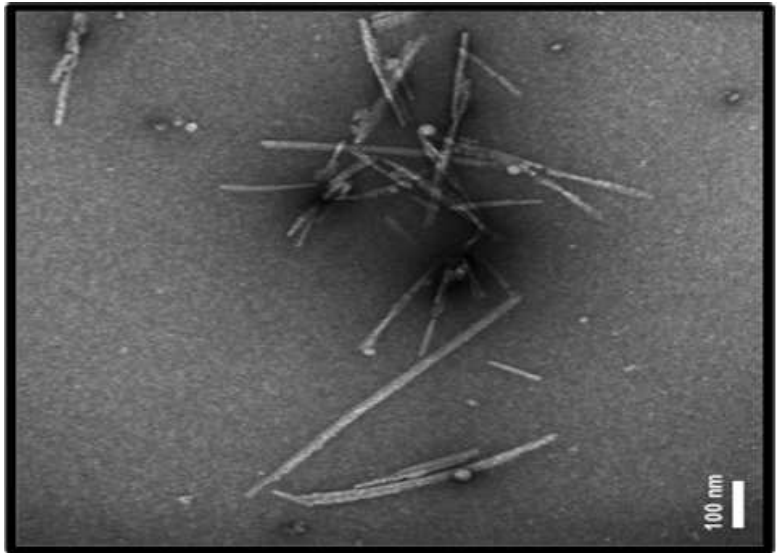}
\caption{View of prion fibrils, Transmission Electron Microscopy image (Courtesy of Prof. J.-P. Liautard, \textsc{inserm} Universit\'e Montpellier 2, France).}
\label{fig1}
\end{center}
\end{figure}
In this paper we deal with idealized rod-like $PrP^{Sc}$, a realistic choice taking into account the flow-related experiments we investigate. We consider the presence of a finite amount of $PrP^C$(free monomers)  and $PrP^{Sc}$ proteins, as well as of ``seeding'' rod-like $PrP^{Sc}$ at initial condition, and fibril lengthening/splitting (\textit{i.e.} fragmentation). It is also important to note that our model is related to \textit{in vitro} experiments: neither source terms of monomers and polymers nor degradation rates are taken into account.

We propose a comprehensive molecular model that accounts for the flow behavior as observed in \emph{in vitro} experiments, focusing on the dynamics of monomers and fibrils. 
A good deal of experimental laboratory work involves complex flows (\textit{e.g}. diffusion, mixing, etc.). Raw data are provided by
sensors designed to acquire macroscopically observable properties like stresses, flow rates, etc. The latter can strongly be influenced by 
the microscopic interactions. Our model does provide an understanding of how various polymers-monomer and polymer-solvent
relationship  result in a configurational probability diffusion equation, with 
the help of which one can investigate the stress tensor and related quantities. 
Therefore, it is of use for flow pattern monitoring sensors.

The current approach is at an early stage of development.  The scission (breakage) process - the most important mechanism in the \emph{in vitro} development/proliferation of infectious proteins - is taken three-dimensionally. While prior models such as those of  \cite{Greer2006,Masel1999} (for mathematically in nature aspects related to, see  \cite{Calvez2009,Engler2006,Greer2007,Laurencot2007,Pruss2006,Scheibel2001,Simonett2006,Walker2007}) neglect the flow influence on prion dynamics, the one in \cite{Greer2006} was rather succesful in predicting prion molecular dynamics in the \emph{in vivo} rest state, and our model is a generalization of \cite{Greer2006}.

The prion fiber is modelled as a rigid rod polymer molecule the length of which is time dependent; see figure \ref{fig2}.

\begin{figure}[htp]
\begin{center}
\includegraphics[width=4in]{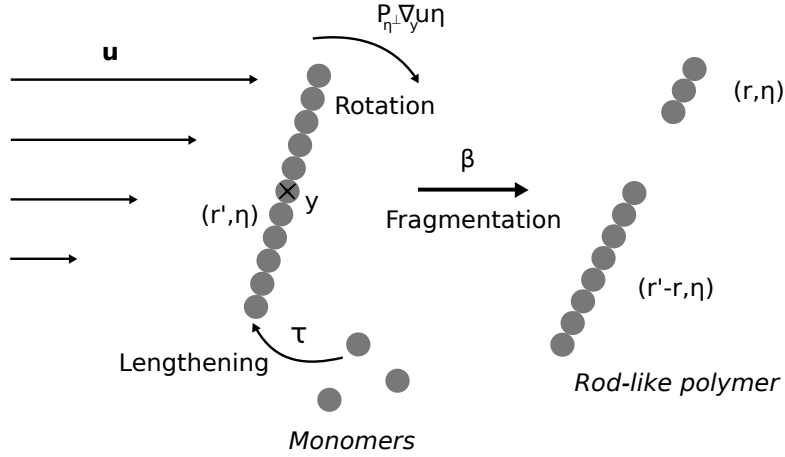}
\caption{Prion fibril modelized as rigid rod polymer under flow.}
\label{fig2}
\end{center}
\end{figure}

The dynamics of rigid rod molecular fluids has been initiated by Kirkwood \cite{Kirkwood1968} and significantly enriched and brought to fruition by Bird and his school \cite{Bird1987} (see also \cite{Huilgol1997} for a more succinct presentation). As in any kinetical theory, the cornerstone is the probability of the configurational diffusion equation, which is of a Fokker-Planck-Smoluchowski type.  The latter is the key ingredient for calculating (the macroscopic) stress tensor and related quantities.  In the following we shall derive a suitable  generalization of equations 14.2-8 in \cite{Bird1987} that account for prion dynamics as observed in experiments \cite{Caughey2009,Prusiner1998,Zomosa-Signoret2007}.

This paper begins by first presenting the constitutive assumptions which later lead to the probability configurational diffusion equation in its general form.  We give a mathematical conceptual framework and a presentation of the main result: the existence of global weak non-negative solution. To achieve this, we obtain a variational formulation of the corresponding boundary value problem, and the proof is based on a semi-discretization in time technique.
The uniqueness of the solution will be proved in a subsequent paper. 
\subsection{The general model}\label{themodel}
%
%
\subsubsection{Polymers}

Let a fiber be modeled as a rod-like molecule here represented by a vector in $\Rd$. For convenience, we use separate symbols for the length $r\in \RR_+=(0,+\infty)$, and for the angle-vector $\eta\in \Sd$, with $\Sd$ being the unit sphere of $\Rd$. Contrary to the  assumption made in \cite{Greer2006} and for simplicity, we assume here that polymers could be arbitrary small, that is no critical (lower) length is considered (this assumption is explained in \cite{Doumic2009}). For technical reasons and without any loss of realistic assumptions, we suppose that fibers are contained in a bounded, smooth open set $\Omega$ in $\Rd$, and the position of each fiber center of mass is denoted by the vector $\y$. We assume a velocity vector field $\uu : \Omega \times \mathbb{R}_+ \rightarrow \mathbb{R}^3$ such that 

\begin{equation}
\nabla_{\y} \cdot \uu = 0 \text{ in } \Omega, \text{ and } \uu\cdot\vec{n} = 0\ \text{on}\ \partial\Omega.
\label{velofield}
\end{equation}
with $\vec n$ the outward normal. The polymer configurational probability distribution function $\psi(r,\eta,\y,t)$, at any time $t>0$, solves the following equation
\begin{equation}
\dfrac{\partial}{\partial t}\psi+ \uu \cdot \nabla_{\y}\psi+ \dfrac{\partial}{\partial r}\left( \tau(\phi,\uu,r,\eta)\psi \right) = \mathcal{B}\psi+\mathcal{F}\psi.
\label{psi}
\end{equation}
with $(r,\eta,\y) \in \RR_+\times \Sd \times \Omega$.
Fibers are transported by the velocity vector field $\uu$ and lengthening occurs at a rate $\tau\geq0$ that depends  on the free monomers density, $\phi$. In dilute regime,  the microscopic hydrodynamics is accounted for by the term $\mathcal{B}$ as in \cite{Otto2008} and defined by
\begin{equation}
[\mathcal{B}\psi](r,\eta,\y,t) = A(r)\ \nabla_\eta \cdot \left[D_1 \nabla_\eta \psi - \drift \ \psi \right],
\end{equation}
where $\nabla_\eta$ and $(\nabla_\eta \cdot)$ denote the gradient and divergence on $\Sd$. $A\geq0$ is a weight function that accounts for the influence of the length increase upon the motion and $D_1>0$ the diffusion coefficient on the sphere. Moreover, the transport on the sphere due to the velocity field is given by $\drift$, with $\mathrm{P}_{\eta^\perp} \z = \z - (\z \cdot \eta)\eta $, for all $\z\in \Rd$, denoting the projection of the vector $\z$ on the tangent space at $\eta$.\\
The fragmentation (scission) process takes place at rate $\beta(\gu,\uu,r,\eta)\geq 0$ and is described by $\mathcal{F}$ following \cite{Greer2006} and
given by
\begin{equation}
[\mathcal{F}\psi](r,\eta,\y,t) = - \beta \psi + 2 \int_r^\infty \beta(\gu,\uu ,r',\eta) \kappa(r,r') \psi (r',\eta,\y,t) \ \dd r'.
\end{equation}

The size redistribution kernel $\kappa$ accounts for the fact that a polymer breaks into smaller fibers. It is symmetric, since a polymer of size $r'$ breaks with equal probability into a fiber of size $r'-r$ and $r$; moreover, the fragmentation/recombination is mass preserving process. We assume here that upon splitting, given the pecularity of the motion process, and its impact on the scission, the resulting clusters of fibrils have the same center of mass as the initial polymer. It seems reasonable to assume that the orientation remains unchanged right after the scission. Therefore: $\kappa(r,r') \geq 0, \; \kappa(r,r')=0 \ \text{if} \ r > r'$, $\kappa(r'-r,r')= \kappa(r,r')$ and
\begin{equation}
\int_0^{r'} \kappa(r,r') \ \dd r = 1.
\label{kappa}
\end{equation}
The probability configurational function $\psi$ must be a non-negative solution, satisfying the non-zero size boundary condition
\begin{equation}
\psi(0,\eta,\y,t) = 0,
\end{equation}
and the initial condition
\begin{equation}
\psi(r,\eta,\y,0) = \psi^0(r,\eta,\y),
\end{equation}
with $\psi^0$ a known non-negative initial probability.
\subsubsection{Monomers}

The concentration of free monomers, given by the distribution $\phi(\y,t)$ at time $t>0$ at any  $\y\in\Omega$, solves
\begin{equation}
\dfrac{\partial}{\partial t}\phi + \uu\cdot\nabla_{\y} \phi - D_2 \Delta \phi =  -  \int_{\Sd\times\RR_+} \tau(\phi,\uu,r,\eta )  \psi(r,\eta,\y,t) \ \dd r \ \dd \eta,
\label{phi}
\end{equation}
with $D_2>0$ the diffusion coefficient. The integral term is due to polymerization of monomers, being transconformed (misfolded), into fibers. Moreover, monomer concentration $\phi$ must be a non-negative solution satisfying the (no transport across) boundary condition
\begin{equation}
\nabla_y \phi \cdot \vec{n} = 0\ \text{on}\ \partial\Omega,
\end{equation}
with $\vec{n}$ the outward normal vector  on the boundary $\partial\Omega$, as well as the initial condition
\begin{equation}
\phi(\y,0) = \phi^0(\y),
\end{equation}
with $\phi^0$ an initially non-negative given concentration.
We adjoin to these equations the balance equation for the total number of monomers contained in the domain $\Omega$:
\begin{equation}
\int_\Omega \left[ \phi(\y,t) + \int_{\RR_+\times\Sd}  r \ \psi(r,\eta,\y,t) \ \dd \eta \ \dd r \right] \ \dd \y =  \rho,\, \mathrm{for\ all}\ t\geq 0,
\label{conservation}
\end{equation}
where $\rho$ is (experimentally) known from the outset.  The above balance equation is formally satisfied, as a consequence of equations \eqref{psi}--\eqref{phi} using also \eqref{velofield}.
%
%
%
%
\subsubsection{Velocity vector field and momentum balance equations}

As an aside, notice the velocity vector field, $\uu(t,\y)\in\Rd$, for all $t>0$ and $\y\in\Omega$, satisfies the Navier-Stokes equations (for incompressible fluids)
\begin{equation}
\left\{\begin{split}
& \deriv{t} \uu + \left(\uu \cdot \nabla\right)\uu  = - \nabla p + \nu \Delta \uu - \nabla \cdot {\bf S},\\
& \nabla \cdot \uu = 0, \\
& \uu \cdot \vec{n} = 0.
\end{split}\right.
\end{equation}
$p$ is the pressure , $\nu$ the viscosity of the Newtonian solvent within which the prions (\emph{i.e.} rigid-rod molecules) are dissolved, and ${\bf S}$ is the non-Newtonian extra stress tensor contribution (to the total stress) due to the presence of rigid rods.  The latter is given by \cite{Bird1987} as
\begin{equation}
{\bf S}(\y,t) = \int_{\RR_+} r^2 \int_{\Sd} \eta  \otimes \eta \ \psi\ \dd \eta  \dd r.
\end{equation}

In this paper, we suppose that $\uu$ is given and the unknown functions are only $\psi$ and $\phi$. The existence and uniqueness of the solutions to the full system with the Navier-Stokes equations introduced above (that is $\uu$, $\psi$ and $\phi$) will be the topic of a subsequent paper.
\subsection{Constitutive assumptions}
\label{gassump}

Assume the velocity vector field satisfies the regularity 
\begin{equation}
\uu\in\mathcal{C}^1\left([0,\infty),W^{1,\infty}(\Omega)\right)
\end{equation}
such that
\begin{equation}
\nabla_{\y} \cdot \uu = 0 \;\;\text{and}\;\;\uu\cdot\vec{n} = 0\ \text{on}\ \partial\Omega.
\label{velocity}
\end{equation}
Next, we adhere to the view on prion proliferation expressed in \cite{Greer2006,Greer2007,Masel1999,Pruss2006}.
The splitting (scission) rate of fibers, given by $\beta$, is assumed to be linear in $r$. Therefore let $g : M_3(\RR)\times\RR^3\times\Sd \rightarrow \RR_+$ be continuous with respect to the first and second variable, such that $\beta(\mathbf{\sigma},\vv,r,\eta) = g(\mathbf{\sigma},\vv,\eta)\ r$, for all $\mathbf{\sigma}\in M_3(\RR)$, $\vv\in\RR^3$, $r>0$ and $\eta\in\Sd$.
Moreover, we assume that for all bounded subsets $B\subset \RR^3$ and $O\subset M_3(\RR)$ there
exist positive constants $\gs_{B,O} \geq \gi_{B,O}$ such that
\begin{equation}
\gi_{B,O} \leq g(\mathbf{\sigma},\vv,\eta) \leq \ \gs_{B,O}, \quad
\mathrm{\;for\;every\;}\ (\mathbf{\sigma},\vv,\eta) \in O\times B \times\Sd.
\end{equation}
Let $T>0$ be fixed.
Then, due to the smoothness of $\uu$, there exists $\gs
\geq \gi > 0$ such that
\begin{equation}
\mathrm{\;for\;every\;}\ (t,\y,\eta)\in [0,T]\times\Omega\times\Sd,\;\;\; \gi \leq g(\gu,\uu,\eta) \leq \gs.
\label{bornu}
\end{equation}
%
%
%
We consider the polymerization rate $\tau$ linear in (the free monomers density) $\phi$,  i.e. there exists $\tau_0 >0$ such that
\begin{equation}
\tau(\phi,\vv,r,\eta) = \tau_0 \phi.
\end{equation}
This assumption had been already evoked by Greer \textit{et al.} \cite{Greer2006} and corresponds to a mass action binding. The splitting kernel $\kappa$ accounts for the probability of a polymer with initial length $r$, to split into a polymer with a shorter length $r'$ as described in \cite{Greer2006}, and is given by
\begin{equation}
\kappa(r,r') =
\begin{cases}
 1/r' & \rm{if}\ 0<r\leq r',\\
 0 & \rm{else.}
\end{cases}
\end{equation}
This expression  is compatible with (\ref{kappa}) (and the conservation law (\ref{conservation})).
Then the length weight function $A\geq 0$ is supposed to be in $L^\infty(\RR_+)$ and there exists $C_A>0$ such that
\begin{equation}
\| A\|_{L^\infty(\Omega)} = C_A <\infty
\label{Anorm}
\end{equation}
We remark that, by virtue of $\uu$ being sufficiently smooth and for fixed $T>0$, there exists $C_P>0$ such that
\begin{equation}
\|\drift \|_{L^\infty([0,T]\times\Omega\times\Sd)} = C_P   < \infty,
\label{driftnorm}
\end{equation}
Using the result stated in the Appendix, there exists $C_D>0$ such that
\begin{equation}
\|\nabla_\eta\cdot\drift\|_{L^\infty([0,T]\times\Omega\times\Sd)} = C_D < \infty.
\label{driftnorm2}
\end{equation}
Thanks to the assumptions given in this section, the problem can be re-written as:
\begin{subequations}
\begin{align} 
\begin{split}
\dfrac{\partial}{\partial t}\psi + \uu \cdot \nabla_{\y}\psi+ \tau_0\phi\dfrac{\partial}{\partial r}\psi - A(r)\ \nabla_\eta \cdot \left[D_1 \nabla_\eta \psi - \drift \ \psi \right]\\
= - g(\gu,\uu,\eta)r \psi + 2 g(\gu,\uu,\eta) \int_r^\infty \psi (r',\eta,\y,t) \ \dd r',
\end{split}\label{psieq}\\
&\dfrac{\partial}{\partial t}\phi + \uu\cdot\nabla_{\y} \phi - D_2 \Delta \phi =  - \tau_0\phi \int_{\Sd\times\RR_+} \psi(r,\eta,\y,t) \ \dd r \ \dd \eta,\label{phieq}\\
&\psi(r=0,\eta,\y,t)=0, \label{boundary1}\\
&\nabla_y \phi\cdot \vec{n}=0, \quad \text{on}\ \partial \Omega \label{boundary2}\\
&\psi(t=0)=\psi^0\;\text{and}\;\phi(t=0)=\phi^0,\label{initial}
\end{align}
\label{system}
\end{subequations}
\subsection{Particular case: zero velocity field, as in the Greer's model}

Consider $\uu = 0$, and assume that $g$ is such that $g(0,\eta) = g_0$, a constant, for any $\eta$.  In fact, even in the absence of flow the prion-fibrils can undergo scission and re-combination.  
Suppose that $\phi$ is independent of $y$,
then let $f(t,r) = \dfrac{1}{|\Omega|}\int_{\Omega\times\Sd} \psi(r,\eta,\y,t)\ \dd \eta
\dd \y$ be the average  of $\psi$.  Integrating  equations (\ref{system}) leads to
\begin{equation}
\left\{
\begin{aligned}
&\deriv{t} f + \tau_0 \phi(t) \deriv{r} f + g_0 r f = 2\gi \int_r^\infty f(r',t)\ \dd r' \;\rm{over}\; (t,r)\in\RR_+^2,\\
& \dfrac{\dd}{\dd t} \phi(t) = - \tau_0 \phi(t) \int_{\RR_+} f(r,t)\ \dd r,\\
&f(0,t) = 0.
\end{aligned}
\right.
\end{equation}
Note that the above system of equations is the one proposed in  \cite{Greer2006} where it was produced under the assumption of prion conservation mass (no  protein synthesis, no metabolic degradation).
\section{Variational formulation and main result}

First we present the functional framework one of the main mathematical novelty of this paper, next  the definition of weak solutions to the system \eqref{system}, and eventually the proof of the existence of a weak solution of this system.
\subsection{Functional framework} \label{functionalframework}

Let $a : \RR_+ \rightarrow \RR_+ $ be defined by $a(r)=\rm{e}^{\alpha r}$ for a $\alpha > 0$. Denote $Q=\Sd\times \RR_+$ and $d\q = a(r) d r d \eta$.  Let the following Hilbert spaces be defined as
\begin{equation}
\La = \left\{ \psi\in L^1_{loc}\left(\OQ \right), \ \int_{\Omega\times Q} \psi^2 \ d\q \dd \y < \infty \right\}.
\end{equation}
Then, 
 \begin{equation}
V = \left\{ \psi\in L^1_{loc}\left(\OQ \right), \ \int_{\OQ} \left(A(r) | \nabla_\eta \psi |^2  + (1+r) \psi^2 \right) \ d\q \dd \y < \infty \right\},
\end{equation}
and
\begin{equation}
\begin{split}V_1 = \bigg\{  \psi\in & L^1_{loc}\left(\OQ \right),\\
&  \int_{\OQ} \left(\left|\deriv{r} \psi\right|^2 + A(r) | \nabla_\eta \psi |^2  + (1+r) \psi^2\right)\ d\q\dd \y < \infty \bigg\}.
\end{split}
\end{equation}
Recall the Sobolev space $H^1(\Omega)$ endowed with the norm
\begin{equation}
\| \phi \|_{H^1} = \| \phi \|_{L^2(\Omega)} + \| \nabla_y \phi \|_{L^2(\Omega)} .
\end{equation}
We also use the canonical embedding
\begin{equation}
V_1 \subset V \subset \La = (\La)^\prime\subset V^\prime \subset \left(V_1\right)^\prime.
\end{equation}
For any $\theta \in \RR $,
let $L^1_\theta = \left\{\psi\in  L^1_{loc}\left(\OQ \right), \; \int_{\OQ} |\psi|\
r^\theta d r d \eta d \y < \infty \right\}$.   Then we have the canonical embedding
\begin{equation}
\La\subset L^1_\theta,\quad \text{for\ any}\; \; \alpha > 0 \;
\text{ and } \; \theta \geq 0,
\end{equation}
which makes sense in regard to the mass conservation and the total quantity of
polymers when $\theta=0$ or $\theta=1$.
\subsection{Variational formulation}

To begin with, we introduce test function spaces.  Let $T>0$.  First, for the polymer $\psi$-equation, let $\mathcal{X}_1$ be the completion of $\mathcal{C}^\infty_c((-T,T)\times\overline{\Omega}\times\Sd\times[0,+\infty))$ with respect to the norm $\|\cdot\|_{\mathcal{X}_1}$
\begin{equation} \|\psit \|_{\mathcal{X}_1} = \int_0^T \left(\left\|\deriv{t}\psit\right\|_{\La}^2 + \|\nabla_{\y} \psit\|^2_{\La} + \|\psit\|_{V_1}^2 \right)\ d t
\end{equation}
In particular, this implies that, if $\psit \in \mathcal{X}_1$, then $\psit(t=T) = 0$.  Second, the
test functions for the $\phi$-equation are elements of $\mathcal{X}_2$, the latter space being the completion of $\mathcal{C}^\infty_c((-T,T)\times\overline{\Omega})$ with respect to the norm $H^1((0,T)\times\Omega)$.  In particular this implies  that if $\phit \in \mathcal{X}_2$, then $\phit(t=T) = 0$.
In order to obtain a variational formulation of (\ref{system}) we first assume that we have a strong solution which is smooth enough. Then we multiply (\ref{psieq}) by $\psit(r,\eta,\y,t) a( r)$, with $\psit\in\mathcal{X}_1$, and integrate over $(0,T)\times \OQ$, next we multiply (\ref{phieq}) by $\phit\in\mathcal{X}_2$ and integrate over $(0,T)\times\Omega$. We note
\begin{equation}
\begin{aligned}
\int_{\RR_+} \tau_0 \phi \deriv{r}\psi\ \psit\ a( r) \dd r  & = - \int_{\RR_+} \tau_0 \phi\psi\ \deriv{r}\left(\psit a( r)\right)\ \dd r,\\
& = - \int_{\RR_+} \tau_0 \phi \psi\ \left(\deriv{r} \psit +  \alpha \psit\right)\ a(r) \dd r, \\
\end{aligned}
\label{vf_derivr}
\end{equation}
since $\psit\in\mathcal{X}_1$.  One also has:
\begin{equation}
\int_{\Sd} \nabla_\eta \cdot \left(D_1 \nabla_\eta \psi\right)\psit\ \dd \eta  = -\int_{\Sd} D_1 \nabla_\eta\psi\cdot\left(\nabla_\eta\psit - 2\eta \psit\right)\ \dd \eta,\\
\end{equation}
and
\begin{equation}
\begin{aligned}
\int_{\Sd} \nabla_\eta \cdot \left( \drift \psi\right) \psit \ d \eta & = - \int_{\Sd}  \drift \psi\cdot \left(\nabla_\eta \psit - 2\eta \psit\right)\ d \eta,\\
& = - \int_{\Sd}  \drift \psi \cdot \nabla_\eta \psit\ d \eta,
\end{aligned}
\end{equation}
since $\drift\cdot \eta =0$ (see for instance Appendix II in
\cite{Otto2008} for calculation details on the sphere). Moreover, by
assumption (\ref{velocity}) on $\uu$,
\begin{equation}
\int_\Omega \left(\uu\cdot\nabla_{\y}\psi\right) \psit\ \dd \y = - \int_\Omega \psi\left(\uu\cdot\nabla_{\y}\psit\right)\ \dd \y,
\end{equation}
and
\begin{equation}
\int_\Omega \left(\uu\cdot\nabla_{\y}\phi\right) \phit\ \dd \y = - \int_\Omega \phi\left(\uu\cdot\nabla_{\y}\phit\right)\ \dd \y.
\end{equation}
Then a variational formulation of (\ref{psieq}) is
\begin{equation}
\begin{aligned}
\begin{split}
& - \int_{\OQ} \psi^0\ \psit(t=0) \ d\q \dd \y  -  \int_0^T\int_{\OQ}\psi \left(\deriv{t}\psit  + \uu \cdot \nabla_{\y} \psit\right)\ d\q\dd \y \ \dd t \\
&+ \int_0^T\int_{\OQ}A(r)\left(D_1 \nabla_\eta\psi \left(\nabla_\eta\psit - 2\eta \psit\right) - \drift\psi \cdot\nabla_\eta\psit\right)\ d\q d\y \ dt\\
&  + \int_0^T\int_{\OQ} \psi \left(g(\gu,\uu,\eta) r  \psit - \tau_0\phi\left(\deriv{r} \psit +  \alpha \psit\right)\right) \ d\q\dd \y  \ \dd t
\end{split}\\
 = 2 \int_0^T\int_{\OQ} g(\gu,\uu ,\eta) \left(\int_r^\infty \psi\ \dd r' \right)\psit \ d\q\dd \y \ \dd t,\\
 \text{for\ any}\ \psit \in \mathcal{X}_1,
\end{aligned}
\label{vfpsi}
\end{equation}
and for (\ref{phieq}),
\begin{equation}
\begin{aligned}
\begin{split}
&-\int_{\Omega}\phi^0\ \phit(t=0)\ \dd\y - \int_0^T\int_{\Omega} \phi \ \left(\deriv{t}\phit + \uu \cdot \nabla_\y \phit\right) \ \dd \y \ \dd t\\
&+ \int_0^T\int_{\Omega} \left[D_2\ \nabla_\y \phi \cdot \nabla_\y \phit  + \tau_0\phi\ \phit \left(\int_{\Sd\times\RR_+} \psi\ \dd r\dd \eta \right)\right]\ \dd \y \ \dd t = 0,
\end{split}\\
\text{for\ any}\ \phit\in \mathcal{X}_2.
\end{aligned}
\label{vfphi}
\end{equation}
\subsection{Main result: existence of non-negative solutions of the problem}

At this point we are prepared to introduce our main result. It gives the existence of non-negative weak solution to our problem under the general assumptions of section \ref{gassump}.
\begin{theorem}[\textbf{Main result}]
Let $\phi^0\in L^\infty(\Omega)$ be non-negative and
$\psi^0 \in \La$ non-negative such that there exists a constant $C_0>0$
with\medskip
\[\psi^0 \leq C_0 e^{-\alpha r}.\medskip\]
Then, for any $T>0$, there exists at least
one solution $(\psi, \phi)$ to the weak formulation
(\ref{vfpsi})-(\ref{vfphi}) of the problem (\ref{system}), with
$\psi$ and $\phi$ non-negative.
Moreover we have $\psi\in L^\infty(0,T;\La)\cap L^2(0,T;V)$ and $\phi\in L^\infty(0,T;L^2(\Omega))\cap L^2(0,T;H^1(\Omega))$.
\bigskip
%
%
\end{theorem}
\begin{remark}
Proving the uniqueness of the solution is a rather lengthy undertaking and will be done in a follow up paper.
\end{remark}

\begin{remark}: Weak solutions to the above variational formulation with stronger
regularity than the one implied by the theorem above satisfy the problem
(\ref{system}) in a strong sense.  Moreover, this variational formulation complies weakly with  the mass conservation principle. Therefore, let  $\varphi \in H^1(0,T)$, with
$\varphi(t=T)=0$, and take $\psit(r,\eta,\y,t)=r e^{-\alpha r}
\varphi(t)\in\mathcal{X}_1$ and $\phit (t,\y)=
\varphi(t)\in\mathcal{X}_2$ in the variational formulations. Using the fact that, for any real value function $f$
\begin{equation}
\int_{\Sd} \eta \cdot \nabla_\eta f \ d \eta = 0.
\end{equation}
we obtain
\begin{equation}
\begin{aligned}
 - \varphi(t=0) \int_\Omega &\left[ \phi^0 + \int_{\RR_+\times\Sd}  r \ \psi^0 \ \dd \eta \ \dd r \right] \dd \y \\
&- \int_0^T \dfrac{d}{d t} \varphi(t)\ \int_\Omega \left[ \phi + \int_{\RR_+\times\Sd}  r \ \psi \ \dd \eta \ \dd r \right] \dd \y\ d t  =  0.
\end{aligned}
\end{equation}
If the solution is smooth enough we have then the mass
conservation result
\begin{equation}
\dfrac{d}{d t} \int_\Omega \left[ \phi + \int_{\RR_+\times\Sd}  r \ \psi \ \dd \eta \ \dd r \right] \ \dd \y = 0.
\end{equation}
\end{remark}
\section{Proof of the main result}

The proof consists of three main steps. First (subsection 3.1), a semi-discretization in time of the problem to obtain an approximation of the solution. Second, we get appropriate estimates (subsection 3.2), and third we obtain a solution by passing to the limit (subsection 3.3).
\subsection{Semi-discretization in time}
%
%
%

Let $N>0$ and $\{t_n\}_{n=0}^N$ a subdivision of $[0,T]$ such that $t_0=0$, $t_N=T$ and $t_n-t_{n-1}=\Delta t >0 $. We
denote by $\psi^n$ and $\phi^n$ the approximations of $\psi$ and
$\phi$ at $t_n$.  Denote $\uu^n(\y) = \uu(t_n,\y)$.
First, for any $s\in[0,T]$, consider the following problem on $[0,T]$:
\begin{equation}
\left\{
\begin{split}
& \dfrac{d}{d t}\chi^n(t) = \uu^n(\chi^n(t)),\\
& \chi(s) = \y.
\end{split}
\right.
\end{equation}
We recall that the regularity of $\uu$ is
%
%
$\mathcal{C}^1(0,T;W^{1,\infty})$,
therefore $\uu^n \in W^{1,\infty}(\Omega)$
so that there exists a unique solution $\chi^n$
which will be denoted in the following by
%
%
$\chi^n(t;s,\y)$. The map $\y \rightarrow \chi^n(t;s,\y)$ is a
homeomorphism from $\Omega$ onto $\Omega$, and since $\uu$ is
divergence-free,  we have
\begin{equation}
\det \nabla_y \chi^n(t;s,\cdot) = 1, \;\;\; \rm{a.e.\ in}\ \Omega\times[0,T].
\end{equation}
Define the function
\begin{equation}
x_n : \Omega \rightarrow \Omega, \;\;\; \text{by}\;\;
\quad x_n(\y) = \chi^n(t_n;t_{n-1}, \y).
\end{equation}
This map $x_n$ is invertible. Let us denote $z_n$ as its inverse.
We remark that
\begin{equation}
z_n(\y) = \chi^n(t_{n-1}; t_n, \y).
\end{equation}
Assume now that $ \psi^{n-1}\in V$ and $\phi^{n-1}\in
H^1$ are known.
We consider two problems:
\\
find $\psi^{n}\in V$ such that
\begin{equation}
\begin{aligned}
\begin{split}
& \int_{\OQ} \dfrac{\psi^n(r,\eta,\y) -\psi^{n-1}(r,\eta,z_n(\y))}{\Delta t}\ \psih\  d\q \dd \y\\
&+  \int_{\OQ} A(r)\left( D_1 \nabla_\eta\psi^n\cdot\left(\nabla_\eta\psih -2\eta\psih\right) -\mathrm{P}_{\eta^\perp}\left(\nabla_{\y} \uu^n \eta\right) \psi^n \cdot\nabla_\eta\psih\right) \ d\q \dd \y\\
&+ \int_{\OQ} \psi^n\left(  g(\gu^n,\uu^n,\eta) r \psih - \tau_0\phi^{n-1} \left(\deriv{r} \psih +  \alpha \psih\right) \right) \ d\q \dd \y 
\end{split}\\
= 2 \int_{\OQ} g(\gu^n,\uu^n
  ,\eta) \left(\int_r^\infty \psi^{n-1}\ \dd r' \right)
\psih \ a(r) \dd r \dd \eta\dd \y,
\end{aligned}
\label{dispsi}
\end{equation}
for\ any $\psih\in V_1,$
and find $\phi^n\in H^1$ such that
\begin{equation}
\begin{aligned}
& \int_{\Omega} \left( \dfrac{\phi^n(\y)-\phi^{n-1}(\y)}{\Delta t}\
+ \uu^n \cdot \nabla_y  \phi^n \right) \phih \
\dd \y \ \dd t  \\
&+ \int_{\Omega} \left[ D_2\ \nabla_\y \phi^n \cdot \nabla_\y \phih  +
  \tau_0\phi^n \ \left(\int_{\Sd\times\RR_+}
    \psi^{n-1}\ \dd r\dd \eta \right) \phih \right] \ \dd \y = 0,\\
\end{aligned}
\label{disphi}
\end{equation}
for\ any $\phih\in H^1$.
Problem (\ref{dispsi}) is re-written as
\begin{equation}
a^n(\psi^n,\psih) = l^n_a(\psih),\;\;\text{for\ any}\ \psih\in V_1
\label{vfpsia}
\end{equation}
with
\begin{equation}
a^n = a^{1 n} + a^{2 n}
\end{equation}
where
$a^{1 n}, \; a^{2 n}$  are defined on $V\times V_1$ by
\begin{equation}
\begin{aligned}
a^{1 n}(\varphi_1,\varphi_2)& = \int_{\OQ}  A(r) D_1\nabla_\eta\varphi_1\cdot \left( \nabla_\eta\varphi_2  -2 \eta \varphi_2\right) \ d\q\dd \y \\
& - \int_{\OQ} A( r) \mathrm{P}_{\eta^\perp}\left(\nabla_{\y} \uu^n \eta \right)\varphi_1 \cdot \nabla_\eta\varphi_2 \ d\q d\y\\
& - \tau_0 \int_{\OQ} \phi^{n-1} \varphi_1 \left(\ \deriv{r} \varphi_2+\alpha\  \varphi_2 \right) \ d\q\dd \y\\
& + \int_{\OQ} g(\gu^n,\uu^n,\eta) r \varphi_1 \varphi_2 \ d\q\dd \y
\end{aligned}
\end{equation}
and
\begin{equation}
a^{2 n}(\varphi_1,\varphi_2) = \dfrac{1}{\Delta t}\int_{\OQ}  \varphi_1 \varphi_2\ d\q \dd \y,
\end{equation}
respectively, and $l^n_a$ is defined on $\La$ by
\begin{equation}
\begin{aligned}
l_a^n(\varphi)=&\ 2\int_{\OQ} \ g(\gu^n,\uu^n,\eta) \left(\int_r^\infty \psi^{n-1}\ \dd r' \right)\varphi \ d\q \dd \y \\
&+ \dfrac{1}{\Delta t} \int_{\OQ} \psi^{n-1}\circ z_n\ \varphi\ d\q\dd \y.
\label{lineara}
\end{aligned}
\end{equation}
The problem (\ref{disphi}) is re-written as
\begin{equation}
b^n(\phi^n,\phih) = l^n_b(\phih),\;\;\text{for\ any}\ \phih\in H^1,
\label{vfphib}
\end{equation}
with $b^n$ defined on $H^1\times H^1$ such that
\begin{equation}
\begin{aligned}
b^n(\varphi_1,\varphi_2) =  & \int_\Omega \left( \dfrac{1}{\Delta
    t} \varphi_1\ \varphi_2\ + \left(\uu^n \cdot \nabla_\y
  \varphi_1\right) \varphi_2 + D_2\ \nabla_\y \varphi_1 \cdot \nabla_\y \varphi_2 \right) \dd \y  \\
&  + \int_\Omega \tau_0\varphi_1 \ \varphi_2\ \left(\int_{\Sd\times\RR_+} \psi^{n-1}\ \dd r\dd \eta \right)\ \dd \y,
 \end{aligned}
\end{equation}
and $l^n_b$ defined on $L^2$ by
\begin{equation}
l_b^n(\varphi) = \dfrac{1}{\Delta t}\int_\Omega \phi^{n-1} \varphi \ \dd \y.
\label{linearb}
\end{equation}
\bigskip
%
%
\begin{lemma}
Let  $N\in\mathbb{N}^*$,  $\phi^0\in L^\infty(\Omega), \; \phi^0 \geq 0 $ \ and \ $\psi^0\in \La$
such that\medskip
\[
0 \leq \psi^0 \leq C_0 e^{- \alpha r} \quad \text {a.e in } \; Q
\medskip\]
with $C_0 > 0$ a constant.
\\
Then there exist two sequences \ $ \{\psi^n\}_{n=1}^N \subset
V $ \  and \ $ \{\phi^n\}_{n=1}^N \subset
H^1(\Omega) $ \ satisfying (\ref{vfpsia}) and (\ref{vfphib}).
\\
Moreover, for $\Delta t$ small enough, we have that:\medskip
\begin{subequations}
\begin{align}
& 0\leq \psi^{n} \leq C_{\infty} e^{-\alpha r}, &  \mathit{\;for\;every\;} \,
n \in  \{0, 1, \cdots N\}, \label{psibornn}\\
& 0 \leq \phi^{n}\leq \|\phi^0\|_{L^\infty}, &  \mathit{\;for\;every\;} \,
n \in  \{0, 1, \cdots N\}, \label{phibornn}
\end{align}
\end{subequations}
and
\begin{equation}
\begin{split}
\begin{split}
& \max_{n = 0, \cdots N} \left[\int_{\OQ} |\psi^n|^2 \ d\q d
\y  + D_1 \Delta t \sum_{n=1}^N
\int_{\OQ} A( r)|\nabla_\eta \psi^n|^2 \ d\q d \y  \right. \\
& \left. + 2 \gi \Delta t \sum_{n=1}^N \int_{\OQ} r |\psi^n |^2\ d\q d \y +
\sum_{n=1}^N \int_{\OQ} |\psi^n - \psi^{n-1}\circ z_n|^2\  d\q d \y \right]
\end{split}\\
 \leq 4 e^{k_3 T} \|\psi^0\|_{L_\alpha^2}^2,
\end{split}
\label{estimlem}
\end{equation}
and
\begin{equation}
\begin{split}
\max_{n = 0, \cdots N} \left[\int_{\Omega} |\phi^n|^2 \, d \y + \sum_{n=1}^N
\int_\Omega |\phi^n - \phi^{n-1}|^2 \, d \y + 2 D_2 \Delta t \sum_{n=1}^N
\int_{\Omega}|\nabla_\y \phi^n |^2\ \dd \y  \right]\\
 \leq 2 \|\phi^0\|_{L^2(\Omega)}^2,
\end{split}
\label{estimlem2}
\end{equation}
where in the above we denoted
\begin{equation}\nonumber
\begin{split}
& k_1= \dfrac{2 \gs}{\alpha},\\
& k_2 =  \alpha\tau_0\|\phi^0\|_{L^\infty} + C_D C_A,\\
& C_\infty = 2 C_0 e^{(k_1 + k_2)T},
\end{split}
\end{equation}
and
\begin{equation}
\nonumber
k_3 =  \alpha\tau_0\|\phi^0\|_{L^\infty} + \frac {C_P^2 C_A} {D_1} + 4
{\bar g} \alpha^{-3/2}  C_{\infty} \sqrt{|\Omega| |S_2|}.
\end{equation}
(Recall $C_D$ and $C_A$ are given by equations
(\ref{Anorm}) and (\ref{driftnorm2})).\\
\label{lemme1}
\end{lemma}
\begin{proof}[Proof of Lemma \ref{lemme1}] $\ $\\
Let us consider the sequence of numbers $\{C_n\}_{n=0}^N$ defined by
induction as
\begin{equation}
C_n = \dfrac{1 + k_1 \Delta t}{1- k_2 \Delta t } C_{n-1}, \quad
\mathrm{\;for\;every\;} \; n = 1, \cdots N.
\label{Cnform}
\end{equation}
with $C_0$ as in the hypothesis of the Lemma.
\\
We proceed by induction. Suppose that $\psi^{n-1}$ and $\phi^{n-1}$
are defined as elements of $V$ and $L^\infty(\Omega) $, respectively.
Suppose also that
\begin{subequations}
\begin{align}
& 0\leq \psi^{n-1} \leq C_{n-1} e^{-\alpha r},  \label{psiborn-n1}\\
& 0 \leq \phi^{n-1}\leq \|\phi^0\|_{L^\infty}.   \label{phiborn-n1}
\end{align}
\end{subequations}
We shall prove the existence of  $\psi^n \in V$ and
$\phi^n \in H^1(\Omega)$
solutions of (\ref{vfpsia}) and (\ref{vfphib}), respectively.
We also prove that they satisfy
\begin{subequations}
\begin{align}
& 0\leq \psi^n \leq C_n e^{-\alpha r},  \label{psiborn-n}\\
& 0 \leq \phi^n\leq \|\phi^0\|_{L^\infty}.   \label{phiborn-n}
\end{align}
\end{subequations}
The above inequalities give \eqref{psibornn} and \eqref{phibornn}
since we have
\begin{equation}
C_n = C_0 \left( \dfrac{1 + k_1 \Delta t}{1- k_2 \Delta t } \right)^n
\leq C_\infty
\end{equation}
for $\Delta t$ small enough.
\medskip \\
\textbf{Step 1.} {\it Regularization and existence.}\medskip\\
We introduce a regularization of $a^n$, denoted $a^n_\varepsilon$ defined on $V_1\times V_1$,
\begin{equation}
a^n_\varepsilon(\varphi_1,\varphi_2) = \varepsilon\int_{\OQ}
\deriv{r}\varphi_1 \deriv{r}\varphi_2\ d\q d \y + a^n(\varphi_1,\varphi_2).
\end{equation}
We shall first prove the existence of a sequence $\left(\psi^n_\varepsilon\right)_\varepsilon$ in $V_1$ solutions of
\begin{equation}
a_\varepsilon^n(\psi^n_\varepsilon,\psih) = l^n_a(\psih),\;\;\;\text{for\ any}\ \psih  \in V_1
\label{aregul}
\end{equation}
Clearly $a^n_\varepsilon$ is bilinear and continuous on
$V_1\times V_1$. Next we prove the coercivity of
$a^n_\varepsilon$. Indeed, let $\varphi\in V_1$ and we remark that
\begin{equation}
\int_{\Sd} 2 \eta \cdot \nabla_\eta \varphi\ \varphi\ d \eta =
\int_{\Sd} \eta \cdot \nabla_\eta \varphi^2 \ d \eta = 0
\end{equation}
since $\nabla_\eta \cdot \eta = 2$ and $\eta\cdot\eta=1$. One has
\begin{equation}
\int_{\Sd} |A( r)\drift \varphi\cdot \nabla_\eta \varphi |\ d \eta \leq \dfrac{1}{2}\int_{\Sd}\left( D_1 A( r) |\nabla_\eta \varphi |^2\ + \dfrac{C_P^2 C_A}{D_1}\varphi^2\right)\ d \eta.
\end{equation}
Finally,
\begin{equation}
\tau_0 \int_{\RR_+}\phi^{n-1} \varphi \deriv{r} \varphi\ a(r) \dd r \leq - \dfrac{1}{2}\alpha \tau_0 \int_{\RR_+}\phi^{n-1} \varphi^2\ a(r) \dd r.
\label{eq-ipp1}
\end{equation}
We remark that this inequality can be proved by using  a regularized sequence
$\left(\varphi_m\right)_m$  that converges to $\varphi$ in $V_1$ and the fact that the remaining term in the right-hand side of \eqref{eq-ipp1} can be dropped according to its appropriate sign. Then, invoking (\ref{phiborn-n1}) and the
above remarks, it follows that
\begin{equation}
\begin{aligned}
a^{1 n}_\varepsilon(\varphi,\varphi) \geq & \dfrac{D_1}{2}\int_{\OQ} A( r)|\nabla_\eta \varphi|^2 \ d\q d \y
+ \gi\int_{\OQ} r \varphi^2\ d\q d \y \\
&- \dfrac{1}{2 D_1}\left(\alpha \tau_0 D_1\|\phi^{0}\|_{L^\infty}+C_P^2 C_A\right)\int_{\OQ} \varphi^2 \ d\q d \y,
\label{coerc-a1}
\end{aligned}
\end{equation}
which in turn implies
\begin{equation}
\begin{aligned}
a^n_\varepsilon(\varphi,\varphi) \geq & \ \varepsilon\int_{\OQ} \left|\deriv{r}\varphi\right|^2 \ d\q d \y + \dfrac{D_1}{2}\int_{\OQ} A( r)|\nabla_\eta \varphi|^2 \ d\q d \y \\
& + \gi\int_{\OQ} r \varphi^2\ d\q d \y \\
&+ \dfrac{1}{2 D_1}\left(\dfrac{2D_1}{\Delta t} - \alpha \tau_0 D_1\|\phi^{0}\|_{L^\infty}-C_P^2 C_A\right)\int_{\OQ} \varphi^2 \ d\q d \y,
\label{coerceps}
\end{aligned}
\end{equation}
The coercivity of $a_\epsilon^n$
follows for $\Delta t$ small enough.
\\
Next, due to the inequality (\ref{psiborn-n1}), we have
\begin{equation}
\int_r^\infty \psi^{n-1}\, d r' \leq \frac {C_{n-1}} \alpha
\end{equation}
which implies that, for any $\varphi \in \La$,
\begin{equation}
\int_{\OQ} g(\gu^n,\uu^n, \eta) \left(\int_r^\infty \psi^{n-1}\
d r'\right)\ |\varphi| \ d\q d \y
\leq \frac {\bar g} \alpha \int_{\OQ} |\varphi|\ d r d \eta d \y.
\label{in-l1}
\end{equation}
One also obtains
\begin{equation}
\int_{\OQ} \psi^{n-1}\circ z_n\ |\varphi|
 \ d\q d \y
\leq C_{n-1} \int_{\OQ} |\varphi|\ d r d \eta d \y.
\label{in-l2}
\end{equation}
We deduce that $l_a^n \in (\La)' \subset (V_1)' $ \ by the continuous
embedding of $\La$ in $L^1$.
%
%
Applying the Lax-Milgram theorem, for all $\varepsilon > 0$ there exists a unique $\psi^n_\varepsilon\in V_1$ solution of (\ref{aregul}).
Next we will prove the existence of  solutions to (\ref{vfphib}). First,
$b^n$ is clearly a bilinear and continuous function on $H^1\times H^1$. To prove its
coercivity, let $\varphi \in H^1$.
Since
\begin{equation}
\int_\Omega \uu^n \cdot \nabla_{\y} \varphi \ \varphi = \frac 1 2
\int_\Omega \uu^n \cdot \nabla_{\y} {\varphi^2} = 0
\label{eq1}
\end{equation}
 we have
\begin{equation}
b^n(\varphi,\varphi) \geq \dfrac{1}{\Delta t}\int_{\Omega} \varphi^2\ \dd \y + D_2 \int_{\Omega}|\nabla_\y \varphi |^2\ \dd \y,
\label{coercb}
\end{equation}
using the positivity of $\psi^{n-1}$, and thus $b^n$ is coercive. Moreover, $l_b^n\in (H^1)'$ since $\phi^{n-1}\in L^\infty$. As a consequence of the Lax-Milgram theorem, there exists a unique $\phi^{n} \in H^1$ satifying (\ref{vfphib}).\medskip\\
\textbf{Step 2. } {\it $L^\infty$ - Estimates}\medskip\\
To begin we first prove two estimates for $\psi^n_\varepsilon$: for its $V$-norm and for its derivative with respect to $r$. It follows from (\ref{coerceps}) and the continuity of $l^n_a$ that there exists a constant $C>0$, dependent of $\Delta t$, such that
\begin{equation}
\begin{aligned}
&\int_{\OQ} \left( A(r)| \nabla_\eta \psi^n_\varepsilon|^2 +  (1+r)|\psi^n_\varepsilon |^2\right)\ d\q d \y \leq C,\\
&\varepsilon  \int_{\OQ} \left|\deriv{r}\psi^n_\varepsilon\right|^2\ d\q \dd \y \leq C.
\label{bornV}
\end{aligned}
\end{equation}
Next we prove the non-negativity of
$\psi^n_\varepsilon$ and $\phi^n$. Let us denote $[\cdot]_+$ and $[\cdot]_-$ respectively the positive and negative part, both positive valued. Then, $\phi^n = [\phi^n]_+ - [\phi^n]_-$ and these two parts belong to $H^1$. We have
\begin{equation}
l^n_b([\phi^n]_-)=b^n(\phi^n,[\phi^n]_-) = -b^n([\phi^n]_-,[\phi^n]_-)
\end{equation}
and invoking (\ref{linearb}) and (\ref{phiborn-n1}), $l^n_b([\phi^n]_-) \geq 0$. Therefore
\begin{equation}
b^n([\phi^n]_-,[\phi^n]_-) \leq 0,
\label{positphi}
\end{equation}
hence $\phi^n \geq0$.
Next, $\psi^n_\varepsilon = [\psi^n_\varepsilon]_+ - [\psi^n_\varepsilon]_-$, the positive and negative parts belong $V_1$, and
\begin{equation}
l^n_a([\psi^n_\varepsilon]_-)=a^n_\varepsilon(\psi^n_\varepsilon,[\psi^n_\varepsilon]_-)=-a^n_\varepsilon([\psi^n_\varepsilon]_-,[\psi^n_\varepsilon]_-),
\end{equation}
Invoking (\ref{lineara}) and (\ref{psiborn-n1}), $l^n_a([\psi^n_\varepsilon]_-)\geq 0$. Thus
\begin{equation}
a^n_\varepsilon([\psi^n_\varepsilon]_-,[\psi^n_\varepsilon]_-) \leq 0,
\label{positpsi}
\end{equation}
hence $\psi^n_\varepsilon\geq 0$.
Let us now obtain $L^\infty$ estimates . We have, according to (\ref{phiborn-n1}) and using the above notation, that
\begin{equation}
\begin{aligned}
b^n([\phi^n- &\|\phi^0\|_{L^\infty}]_+,[\phi^n-\|\phi^0\|_{L^\infty}]_+) \\
&= b^n(\phi^n-\|\phi^0\|_{L^\infty}, [\phi^n-\|\phi^0\|_{L^\infty}]_+) \\
& =  b^n(\phi^n,[\phi^n- \|\phi^0\|_{L^\infty}]_+) - b^n(\|\phi^0\|_{L^\infty},[\phi^n- \|\phi^0\|_{L^\infty}]_+)\\
& = l^n_b([\phi^n- \|\phi^0\|_{L^\infty}]_+)- b^n(\|\phi^0\|_{L^\infty},[\phi^n- \|\phi^0\|_{L^\infty}]_+)\\
& \leq \dfrac{1}{\Delta t} \int_\Omega \left( \phi^{n-1} - \|\phi^0\|_{L^\infty} \right) [\phi^n- \|\phi^0\|_{L^\infty}]_+ \ d \y,
\end{aligned}
\end{equation}
Then by (\ref{phiborn-n1})
\begin{equation}
b^n([\phi^n- \|\phi^0\|_{L^\infty}]_+ ,[\phi^n-\|\phi^0\|_{L^\infty}]_+) \leq 0,
\label{inftybphi}
\end{equation}
hence $\phi^n \leq \|\phi^0\|_{L^\infty}$.
Let $C_n$ as defined in (\ref{Cnform}); then
\begin{equation}
\begin{aligned}
 a^n_\varepsilon([\psi^n_\varepsilon-& C_n e^{-\alpha  r}]_+,[\psi^n_\varepsilon-C_n e^{-\alpha r}]_+)\\
 & =
 a^n_\varepsilon(\psi^n_\varepsilon-C_n e^{-\alpha
  r},[\psi^n_\varepsilon-C_n e^{-\alpha r}]_+)
\\
& = a^n_\varepsilon(\psi^n_\varepsilon,[\psi^n_\varepsilon-C_n e^{-\alpha r}]_+) - a^n_\varepsilon(C_n e^{-\alpha r},[\psi^n_\varepsilon-C_n e^{-\alpha r}]_+)\\
& = l^n_a([\psi^n_\varepsilon-C_n e^{-\alpha r}]_+) - a^n_\varepsilon(C_n e^{-\alpha r},[\psi^n_\varepsilon-C_n e^{-\alpha r}]_+).
\end{aligned}
\label{est1-ae}
\end{equation}
Next, for any $\varphi \in V_1$ positive,
\begin{equation}
\begin{aligned}
a^n_\varepsilon(C_n e^{-\alpha r},\varphi) = & -\varepsilon \int_{\OQ} \alpha C_n \deriv{r} \varphi \ \dd r \dd \eta\dd \y\\
& - \int_{\OQ}C_n  A(r) \mathrm{P}_{\eta^\perp}\left(\nabla_{\y} \uu^n \eta \right) \cdot \nabla_\eta \varphi \ \dd r  \dd \eta\dd \y\\
& - \int_{\OQ} C_n \tau_0 \phi^{n-1} \left(\ \deriv{r} \varphi +\alpha\  \varphi \right) \ \dd r  \dd \eta\dd \y \\
& + \int_{\OQ}  C_n g(\gu^n,\uu^n,\eta) r  \varphi \dd r \dd \eta\dd \y +  C_n \dfrac{1}{\Delta t}\int_{\OQ}   \varphi \ \dd r \dd \eta\dd \y.
\end{aligned}
\end{equation}
We remark that
\begin{equation}
\begin{aligned}
&\varepsilon \int_{\OQ} \alpha C_n \deriv{r} \varphi \ \dd r \dd \eta\dd \y = - \varepsilon \int_{\Omega\times \Sd}   \alpha C_n\varphi(r=0,\eta,\y)\ \dd \eta\dd \y  \leq 0,\\
& \int_{\OQ} C_n \tau_0 \phi^{n-1}  \deriv{r} \varphi \ \dd r  \dd \eta\dd \y = -   \int_{\Omega\times \Sd}  C_n\tau_0 \phi^{n-1} \varphi(r=0,\eta,\y)\ \dd \eta\dd \y  \leq 0.\\
\end{aligned}
\end{equation}
Then, by (\ref{Anorm}), (\ref{driftnorm2}), (\ref{phiborn-n1}) and
the positivity of $\varphi$,
\begin{equation}
\begin{aligned}
a^n_\varepsilon(C_n e^{-\alpha r},\varphi) & \geq \int_{\OQ}C_n  A(r) \nabla_\eta\cdot\left(\mathrm{P}_{\eta^\perp}\left(\nabla_{\y} \uu^n \eta \right)\right) \varphi \ \dd r  \dd \eta\dd \y\\
& +C_n\left(\dfrac{1}{\Delta t}- \alpha\tau_0\|\phi^0\|_{L^\infty}\right)\int_{\OQ} \varphi \ \dd r  \dd \eta\dd \y\\
&\geq C_n\left(\dfrac{1}{\Delta t} - k_2
\right)\int_{\OQ} \varphi \ \dd r  \dd \eta\dd \y.
\end{aligned}
\label{est2-ae}
\end{equation}
Moreover, by (\ref{lineara}), (\ref{in-l1}) and (\ref{in-l2})
\begin{equation}
l^n_a(\varphi) \leq  C_{n-1}\left(\dfrac{2 \gs}{\alpha} +
  \dfrac{1}{\Delta t}\right) \int_{\OQ} \varphi\ d r d \eta d \y.
\label{est1-la}
\end{equation}
Now, replacing $\varphi$ by $[\psi^n_\varepsilon-C_n\rm{e}^{-\alpha
  r}]_+$ and using  \eqref{est1-ae} \eqref{est2-ae}  and
\eqref{est1-la} one gets
\begin{equation}
\begin{split}
a^n_\varepsilon([\psi^n_\varepsilon- & C_n e^{-\alpha
  r}]_+,[\psi^n_\varepsilon-C_n e^{-\alpha r}]_+) \\
  & \leq
\left [ C_{n-1} \left( k_1 + \frac 1 {\Delta t} \right) -
 C_n \left(\frac 1 {\Delta t} - k_2 \right) \right] \int_{\OQ} \varphi\
d r d \eta d \y.
\end{split}
\end{equation}
Using now the particular form of $C_n$ gives
\begin{equation}
a^n_\varepsilon([\psi^n_\varepsilon-C_n e^{-\alpha r}]_+,[\psi^n_\varepsilon-C_n e^{-\alpha r}]_+) \leq 0,
\label{psiinftyborn}
\end{equation}
hence
\begin{equation}
\psi^n_\varepsilon \leq C_n e^{-\alpha r}.
\label{pnb1}
\medskip
\end{equation}
\textbf{Step 3. } {\it Convergence and  positivity} \medskip\\
%
The sequence $\left(\psi_\varepsilon^n\right)_\varepsilon$ obtained for all $\varepsilon>0$ is uniformly bounded in $V$  by
(\ref{bornV}), so it weakly converges to an element $\psi^n\in V$ up to a
subsequence. Moreover, $\left(\varepsilon^{1/2} \frac{\partial}{\partial
  r}\psi_\varepsilon^n\right)_\varepsilon$ is bounded in $\La$, then for $\varepsilon \rightarrow 0$, $\psi^n$ solves (\ref{vfpsia}).The positivity  of $\psi^n_\varepsilon$ yields the positivity of $\psi^n$. Moreover, by virtue of \eqref{pnb1},  $\psi^n$ for $\varepsilon \rightarrow 0$, and inequalities \eqref{psibornn} are satisfied.
%
\medskip\\
\textbf{Step 4. } {\it Additional estimates}\medskip
\\
From (\ref{coerceps}), (\ref{lineara}) and (\ref{psibornn}) one gets
\begin{equation}
\begin{aligned}
0 = a^n_\varepsilon(\psi^n_\varepsilon,\psi^n_\varepsilon) - l^n_a(\psi^n_\varepsilon) \geq & \dfrac{D_1}{2}\int_{\OQ} A( r)|\nabla_\eta \psi^n_\varepsilon|^2 \ d\q d \y \\
& + \gi\int_{\OQ} r |\psi^n_\varepsilon|^2\ d\q d \y \\
& - \frac {k_3} 2
\int_{\OQ} |\psi^n_\varepsilon|^2 \ d\q d \y\\
& + \dfrac{1}{\Delta t} \int_{\OQ}\left(\psi^n_\varepsilon - \psi^{n-1}
\circ z_n \right) \psi^n_\varepsilon\ d\q d \y.
\end{aligned}
\end{equation}
Remarking that $2 s_1(s_1-s_2) = s_1^2 + (s_1-s_2)^2 - s_2^2$ for any
reals $s_1, s_2$, leads to
\begin{equation}
\begin{aligned}
\begin{split}
& D_1 \int_{\OQ} A( r)|\nabla_\eta \psi^n_\varepsilon|^2 \ d\q d \y  + 2 \gi\int_{\OQ} r |\psi^n_\varepsilon|^2\ d\q d \y \\
& + \dfrac{1}{\Delta t} \int_{\OQ} \left[ |\psi^n_\varepsilon|^2  +
  |\psi^n_\varepsilon - \psi^{n-1} \circ z_n|^2 -
  |\psi^{n-1} \circ z_n |^2 \right] \ d\q d \y
\end{split}\\
 \leq k_3 \int_{\OQ} |\psi^n_\varepsilon|^2 \ d\q d \y.
\end{aligned}
\end{equation}
Then, taking the $\liminf$ for  $\varepsilon\rightarrow 0$,
multiplying by $\Delta t$ and  using the fact that
\\
$\int_\Omega |\psi^{n-1} \circ z_n|^2 = \int_\Omega |\psi^{n-1}|^2$,
gives
\begin{equation}
\begin{aligned}
\begin{split}
& D_1 \Delta t \int_{\OQ} A( r)|\nabla_\eta \psi^n|^2 \ d\q d \y  + 2 \gi \Delta t \int_{\OQ} r |\psi^n|^2\ d\q d \y \\
& + (1 - k_3 \Delta t)  \int_{\OQ}|\psi^n|^2 \ d\q d \y +
 \int_{\OQ}
  |\psi^n - \psi^{n-1} \circ z_n|^2  \ d\q d \y
  \end{split}\\
 \leq \int_{\OQ} |\psi^{n-1}|^2 \ d\q d \y.
\end{aligned}
\end{equation}
Multiply the last inequality by $(1 - k_3 \Delta t)^{n-1}$ and sum over $n$ from $n=1$ to $n=N$. Use the inequality
$$
(1 - k_3 \Delta t)^n \geq (1 - k_3 \Delta t)^N \geq
\frac 1 2 e^{- k_3 T}
  $$
to get \eqref{estimlem}.
Taking $\hat{\phi} = \phi^n$ in \eqref{disphi} and using
\eqref{phibornn} and \eqref{eq1}
we obtain
%
\begin{equation}
\dfrac{1}{2 \Delta t}\int_{\Omega}\left(|\phi^n|^2 +|\phi^n -
  \phi^{n-1}|^2 - |\phi^{n-1}|^2\right) \ \dd \y +
D_2 \int_{\Omega}|\nabla_\y \phi^n |^2\ \dd \y \leq 0
\end{equation}
Summing over $n$ from 1 to $N$ produces (\ref{estimlem2}), which
ends the proof. \bigskip
\end{proof}
\subsection{Construction of a solution}
%
%
We now define, for any $N$ large enough,  the following functions
\begin{equation}
\psi_N(\cdot,t) = \dfrac{t-t_{n-1}}{\Delta
  t}\psi^n(\cdot)+\dfrac{t_n-t}{\Delta t}\psi^{n-1}(\cdot),
\;\;\; t \in[t_{n-1},t_n]
\end{equation}
and
\begin{equation}
\psi_N^+(\cdot,t) = \psi^n(\cdot),\;\;\; \psi_N^{-}(\cdot,t) = \psi^{n-1}(\cdot),\;\;\; t \in(t_{n-1},t_n]
\end{equation}
for $ n = 1, \cdots N$.
\\
We shall use analogous notations for $\phi_N$ and $\uu_N$.
Let $\psit\in \mathcal{X}_1, \; \phit\in \mathcal{X}_2$, both be test functions \ and set
$\psih  = \int_{t_{n-1}}^{t_n} \psit \, dt$ \ and \
$ \phih = \int_{t_{n-1}}^{t_n} \phit \, dt$.
It is clear that $ \psih \in V_1 $ and $ \phih \in H^1(\Omega)$. Then
%
%
\begin{equation}
\begin{aligned}
& \int_{t_{n-1}}^{t_n} a^n(\psi^n, \psit(\cdot, t)) \, dt =
\int_{t_{n-1}}^{t_n} l_a^n(\psit(\cdot, t)) \, dt,\\
& \int_{t_{n-1}}^{t_n} b^n(\psi^n, \psit(\cdot, t)) \, dt =
\int_{t_{n-1}}^{t_n} l_b^n(\psit(\cdot, t)) \, dt.
\end{aligned}
\end{equation}
Adding these inequalities, we obtain, for any $\psit \in \mathcal{X}_1$,
\begin{equation}
\begin{aligned}
\begin{split}
& \int_0^T\int_{\OQ} \dfrac{\psi_N^{+}(r,\eta,\y,t) -
\psi_N^{-}(r,\eta,\z_N(\y, t),t)}{\Delta t}\ \psit(r,\eta,\y,t)\  d\q\dd \y\\
&+ D_1 \int_0^T \int_{\OQ} A(r) \nabla_\eta\psi_N^{+}\cdot\left(\nabla_\eta\psit - 2\eta \psit \right) \ d\q\dd \y\\
&- \int_0^T \int_{\OQ} A(r) \mathrm{P}_{\eta^\perp}\left(\nabla_{\y} \uu_N^{+} \eta \right)\psi_N^{+} \cdot \nabla_\eta \psit \ d\q\dd \y\\
&+ \int_0^T\int_{\OQ}\psi_N^{+} \left( g(\gu_N^+, \uu_N^{+},\eta) r \psit - \tau_0\phi_N^{-}\left(\deriv{r} \psit +  \alpha \psit\right) \right) \ d\q\dd \y
\end{split}\\
= 2 \int_0^T\int_{\OQ}  g(\gu_N^+, \uu_N^{+} ,\eta) \left(\int_r^\infty \psi_N^{-}\ \dd
    r' \right)\psit \ d\q\dd \y,
\label{e-psn}
\end{aligned}
\end{equation}
where in the above,
\begin{equation}
\x_N(\y,t)= x_n(\y) \quad \text{and} \quad
\z_N(\y,t)= z_n(\y), \quad \text{for\ any } \; t \in \, (t_{n-1},t_n).
\end{equation}
Proceeding likewise, for any $\phit \in \mathcal{X}_2$,
\begin{equation}
\begin{aligned}
\begin{split}
& \int_0^T\int_{\Omega}
\dfrac{\phi_N^{+}(\y,t)-\phi_N^{-}(\y,t)}{\Delta t}\ \phit(\y,t) \ \dd
\y \dd t +  \int_0^T\int_{\Omega} \left(u_N^+ \cdot \nabla_y \phi_N^{+}\right)
\phit \ \dd \y \dd t \\
&+ \int_0^T\int_{\Omega} D_2\ \nabla_\y \phi_N^{+} \cdot \nabla_\y \phit \ d \y d t
+ \tau_0 \int_0^T\int_{\Omega}\phi_N^{+}  \left(\int_{\Sd\times\RR_+} \psi_N^{-}\ \dd r\dd \eta \right)\phit \ \dd \y d t\\
\end{split}\\
= 0.
\label{e-phn}
\end{aligned}
\end{equation}
However, to evaluate the limit $\Delta t \rightarrow 0$, we need some additional convergence results about the approximations. First, let us define the maps,
\begin{equation}
\begin{aligned}
&\Lambda_1 [\psi] (\y,t) = \int_{\Sd\times \RR_+} \psi(r,\eta,\y,t) \ d r d \eta,&\\
&\Lambda_2 [\psi](r,\eta,\y,t) = \int_r^\infty \psi(r',\eta,\y,t) \ d r',&\text{for\ any}\ \psi\in L^2(0,T;\La).
\end{aligned}
\end{equation}
We have the following lemma:
\begin{lemma}
Let   $\phi^0\in L^\infty(\Omega), \; \phi^0 \geq 0 $ \ and \ $\psi^0\in \La$
such that\medskip
\[
0 \leq \psi^0 \leq C_0 e^{- \alpha r} \quad \text {a.e in } \; Q
 \medskip\]
with $C_0 > 0$ a constant. For $\left\{\psi_N\right\}_N$ and $\left\{\psi_N^\pm\right\}_N$, constructed by virtue of Lemma \ref{lemme1}, there exists
$\psi \in L^2(0,T;V)\cap L^\infty(0,T;\La)$, positive, such that, for $ N \rightarrow + \infty$ we have the following convergence, up to a
subsequence of $N$:
\begin{align}
&\psi_N^{\pm} \rightharpoonup \psi & \ * - weakly\ in\ L^\infty(0,T;\La),\label{psic}\\
&A^{1/2} \nabla_\eta\psi_N^{+}\rightharpoonup A^{1/2}  \nabla_\eta \psi & weakly\ in\ L^2(0,T;\La),\label{dpsic}\\
&r^{1/2} \psi_N^{+}\rightharpoonup r^{1/2} \psi & weakly\ in\ L^2(0,T;\La),\label{rpsic}\\
&\Lambda_1 [\psi_N^-] \rightharpoonup \Lambda_1 [\psi] & weakly\ in\ L^2((0,T)\times\Omega),\label{lbd1c}\\
&\Lambda_2 [\psi_N^-]  \rightharpoonup \Lambda_2 [\psi]& weakly\ in\ L^2(0,T;\La).\label{lbd2c}
\end{align}
\label{convergence1}
\end{lemma}
\begin{proof}
%
%
It is clear from  \eqref{estimlem} that
\begin{equation}
\psi_N^+ \quad \text{ is bounded in } \; L^2(0,T;V)
\end{equation}
and
\begin{equation}
\psi_N^{\pm} \quad \text{ is bounded in } \; L^\infty(0,T;L_\alpha^2).
\label{bd-psm}
\end{equation}
We then deduce that
\begin{equation}
\psi_N^- \circ z_N \quad \text{ is bounded in } \; L^\infty(0,T;L_\alpha^2).
\end{equation}
From  \eqref{estimlem} one infers
\begin{equation}
\psi_N^+ - \psi_N^- \circ z_N  \rightarrow 0 \quad \text{in the norm of } \; \;
L^2(0,T;L_\alpha^2).
\end{equation}
Then there exists $ \psi^+ \in L^2(0,T;V) \cap L^\infty(0,T;L_\alpha^2)$ and
$\psi^- \in L^\infty(0,T;L_\alpha^2)$ \ such that, up to a subsequence in $N$
we have
\begin{equation}
\psi_N^+ \rightharpoonup \psi^+ \quad weakly\ in \; L^2(0,T;V)
\end{equation}
\begin{equation}
\psi_N^{\pm} \rightharpoonup \psi^{\pm} \quad *- weakly\ in \;
L^\infty(0,T;L_\alpha^2),
\end{equation}
and
\begin{equation}
\psi_N^- \circ z_N \rightharpoonup  \psi^+ \quad *-weakly\ in \;
L^\infty(0,T;L_\alpha^2).
\end{equation}
On the other hand we have
\begin{equation}
\begin{aligned}
x_n(\y) - \y & = \chi^n(t_n ; t_{n-1} ,\y) - \chi^n (t_{n-1};t_{n-1},\y)\\
& = \Delta t\ \frac {\partial }{\partial t} \chi^n (\xi;t_{n-1},\y) \\
& = \Delta t\ \uu^n(\chi^n (\xi;t_{n-1},\y)).
\end{aligned}
\end{equation}
This implies
\begin{equation}
\left\| \x_N(\y, t) - \y \right\|_{L^\infty(]0,T[ \times\Omega)} \leq \
\Delta t \, \left\| \uu \right\|_{L^\infty(]0,T[ \times\Omega)}
\label{sm-dif}
\end{equation}
%
Now, for any $\psit \in \mathcal{C}^\infty_0(Q\times \Omega \times ]0,T[)$, with the help of \eqref{sm-dif} and \eqref{bd-psm}, we obtain
\begin{equation}
\begin{aligned}
 &\left| \int_0^T \int_{\OQ} \left[\psi_N^{-}(r,\eta,\y,t) - \psi_N^{-}(r,\eta,\z_N(\y,t),t)\right] \psit(r,\eta,\y, t)\ d\q d \y dt \right| \\
& =\left| \int_0^T \int_{\OQ} \psi_N^{-}(r,\eta,\y,t) \left[\psit(r,\eta,\y, t) - \psit(r,\eta,\x_N(\y,t), t)\right]\ d\q d \y dt \right|\\
&\leq C \Delta t \left\| \uu \right\|_{L^\infty([0,T]\times\Omega)} \|\psit\|_{C^1}.
\end{aligned}
\end{equation}
We deduce  that \
$ \psi_N^- - \psi_N^- \circ z_N  \rightarrow 0 $ \ in the sense of distributions \ $\mathcal{D}'(Q \times ]0,T[)$. This leads to the conclusion that
 $\psi^+ = \psi^-$, and we denote
by $\psi$ the common value $\psi^+$ or $\psi^-$.
Therefore \eqref{psic}, \eqref{dpsic} and \eqref{rpsic} are proved.
%
%
Let now $\varphi \in L^2((0,T)\times\Omega)$
\begin{equation}
\begin{aligned}
&\int_0^T \int_{\Omega}\left(\Lambda_1 \psi_N^{-} - \Lambda_1 \psi\right) \varphi(\y,t)\ d \y d t\\
 & = \int_0^T \int_{\OQ} \psi_N^{-}  \varphi e^{-\alpha r} \ d\q d \y d t - \int_0^T \int_{\OQ} \psi  \varphi e^{-\alpha r} \ d\q d \y d t\\
 &\rightarrow 0,\;\;\;\text{as}\; N \rightarrow + \infty
\end{aligned}
\end{equation}
since $\varphi e^{-\alpha r} \in \La$. Now, invoking (\ref{psic}), proves (\ref{lbd1c}). Finally, let $\psit\in L^2(0,T;\La)$ and with the help of (\ref{rpsic}) we get
\begin{equation}
\begin{aligned}
& \int_0^T \int_{\OQ} \left(\Lambda_2 \psi_N^{ +}-\Lambda_2 \psi \right) \psit \ d\q d \y d t \\
& =  \int_0^T \int_{\OQ} r \psi_N^{+} \psit \ d\q d \y d t -  \int_0^T \int_{\OQ} r \psi  \psit \ d\q d \y d t \\
& \rightarrow 0,\;\;\;\text{as}\;\; N \rightarrow +\infty.
\end{aligned}
\end{equation}
Which proves (\ref{lbd2c}).
The positivity of $\psi$ follows from the positivity of $\psi^n$ for any $n$. This ends the proof.
\medskip
\end{proof}
We now focus on the convergence of the $\phi_N$ sequence.
\begin{lemma}
Let   $\phi^0\in L^\infty(\Omega), \; \phi^0 \geq 0 $ \ and \ $\psi^0\in \La$
such that \medskip
\[
0 \leq \psi^0 \leq C_0 e^{- \alpha r} \quad \text {a.e in } \; Q
\medskip\]
with $C_0 > 0$ a constant. For $\left\{\phi_N\right\}_N$ and $\left\{\phi_N^\pm\right\}_N$, constructed by virtue of Lemma \ref{lemme1},
there exists $\phi \in L^2(0,T;H^1)\cap L^\infty(0,T;L^2)$ positive such that  we have the following convergence, up to a subsequence of $N$:
\begin{align}
&\nabla_{\y} \phi_N^{+} \rightharpoonup \nabla_{\y} \phi\ &  weakly \  L^2(0,T;L^2)\label{dphic1}\\
& \phi_N^{\pm}, \; \phi_N \rightarrow \phi & strongly\ L^2(0,T;L^2(\Omega)) \label{strong}
\end{align}
\label{convergence2}
\end{lemma}
\begin{proof}
From (\ref{estimlem2}), we deduce that
\begin{equation}
\phi_N^+ \quad \text{ is bounded in } \; L^2(0,T;H^1(\Omega)),
\label{bl1}
\end{equation}
\begin{equation}
\phi_N^{\pm} \quad \text{ is bounded in } \; L^\infty(0,T;L_\alpha^2)
\label{bl2}
\end{equation}
and
\begin{equation}
\phi_N^- \quad \text{ is bounded in } \; L^2(\delta, T; H^1(\Omega)) \; \;
\text{ for any } \; \; \delta \in ]0, T[.
\label{bl3}
\end{equation}
Since we have
$$
\phi_N = \frac {t_n - t} {\Delta t} \phi_N^- +
\frac {t - t_{n-1}} {\Delta t} \phi_N^+
   $$
we deduce that
\begin{equation}
\phi_N \quad \text{ is bounded in } \; L^\infty(0,T;L_\alpha^2)
\label{bl4}
\end{equation}
and
\begin{equation}
\phi_N \quad \text{ is bounded in } \; L^2(\delta, T; H^1(\Omega)) \; \;
\text{ for any } \; \; \delta \in ]0, T[.
\label{bl5}
\end{equation}
It follows there exists a $\phi \in L^2(0,T;H^1)\cap L^\infty(0,T;L^2)$ such that
\eqref{dphic1} is satisfied.
On the other hand, from the equality
\begin{equation}
\frac{\partial \phi_N} {\partial t} = \frac {\phi^n - \phi^{n-1}} {\Delta t} \quad
\text{ on } \; \; [t_{n-1}, \; t_n]
\end{equation}
and from \eqref{disphi} we deduce that for any $\phih \in H^1(\Omega)$ we have
\begin{equation}
\begin{split}
\int_\Omega \dfrac{\partial \phi_N} {\partial t} \phih\ \dd \y = &
- \int_\Omega u_N^+ \cdot  \nabla_{\y}\phi_N^{+} \phih \ d \y  - D_2 \int_\Omega \nabla_{\y}\phi_N^{+}\cdot\nabla_{\y}\phih\ \dd \y\\
 & - \tau_0 \int_{\Omega} \phi_N^{+} \left(\int_{\Sd\times\RR_+}\psi_N^{-}\ d r d \eta \right) \phih \ d \y
\end{split}
\end{equation}
Using \eqref{bl1} and \eqref{bd-psm}, gives
\begin{equation}
 \frac {\partial \phi_N} {\partial t} \quad \text{ is bounded in } \;
L^2(0, T ; (H^1(\Omega))').
\label{bl6}
\end{equation}
Then, up to a subsequence of $N$, we have
\begin{equation}
\phi_N \rightarrow \phi \quad \text{ strongly in } \;
L^2(\delta , T ; L^2(\Omega)), \quad \text{ for any } \; \; \delta \in \, ]0, T[.
\label{bl7}
\end{equation}
Let us now prove that
\begin{equation}
\phi_N \rightarrow \phi \quad \mathrm{strongly\ in}  \;
L^2(0, T ; L^2(\Omega)).
\label{bl8}
\end{equation}
We fix $\varepsilon > 0$ and we have for any $\delta \in \, ]0, T[$:
\begin{equation}
\begin{aligned}
\int_0^T \| \phi_N - \phi \|_{L^2(\Omega)}^2 \, dt & =
\int_0^{\delta} \| \phi_N - \phi \|_{L^2(\Omega)}^2 \, dt +
\int_{\delta}^T \| \phi_N - \phi \|_{L^2(\Omega)}^2 \, dt\\
& \leq 2 C \delta + \int_{\delta}^T \| \phi_N - \phi \|_{L^2(\Omega)}^2 \, dt
\end{aligned}
\end{equation}
where $C$ is an upper bound for $\|\phi_N\|_{L^\infty(0, T ; L^2)}$ and
$\|\phi\|_{L^\infty(0, T ; L^2)}$.
Now taking $ \delta = \frac {\varepsilon} {4 C} $ we obtain from \eqref{bl7} that
for $N$ large enough
\begin{equation}
\int_{\delta}^T \| \phi_N - \phi \|_{L^2(\Omega)}^2 \, dt \leq \frac \varepsilon 2,
\end{equation}
which proves \eqref{bl8}.
From \eqref{estimlem2} one gets
\begin{equation}
\phi_N^+ - \phi_N^- \rightarrow 0 \quad \mathrm{strongly\ in}  \; \;
L^2(0, T; L^2(\Omega)).
\end{equation}
Using the fact that
\begin{equation}
\phi_N - \phi_N^+ = \frac {t - t_n} {\Delta t} (\phi_N^+ - \phi_N^-)
\end{equation}
and
\begin{equation}
\phi_N - \phi_N^- = \frac {t - t_{n-1}} {\Delta t} (\phi_N^+ - \phi_N^-)
\end{equation}
leads to
\begin{equation}
\phi_N - \phi_N^{\pm} \rightarrow 0 \quad \mathrm{strongly\ in}  \; \;
L^2(0, T; L^2(\Omega)).
\end{equation}
This ends the proof.
\bigskip
\end{proof}

\subsection{Final stage of the proof of the main result}

In the following we let $N \rightarrow + \infty$ in \eqref{e-psn} and \eqref{e-phn}
with $\psit \in \mathcal{C}^\infty_c((-T,T)\times\overline{\Omega}\times\Sd\times[0,+\infty))$ and
$\phit \in \mathcal{C}^\infty_c((-T,T)\times\overline{\Omega}\times\Sd\times[0,+\infty))$, respectively.
We now prove that $\psi$ and $\phi$ given by Lemmas \ref{convergence1} and \ref{convergence2}  satisfy the  variational equalities \eqref{vfpsi}  and \eqref{vfphi}, respectively.
%
%
%
%
%
%
Since $\Delta t$ is small enough, we have
\begin{equation}
\begin{aligned}
& \int_0^T\int_{\OQ} \dfrac{\psi_N^{+}(r,\eta,\y,t) -\psi_N^{-}(r,\eta,\z_N(\y,t),t)}{\Delta t}\ \psit(r,\eta,\y,t)\  d\q\dd \y d t\\
& = - \int_0^T\int_{\OQ} \psi_N^{-}(r,\eta,\y,t) \dfrac{\psit(r,\eta,\x_N(\y,t),t) -\psit(r,\eta,\y,t-\Delta t)}{\Delta t}\  d\q\dd \y d t \\
& \; \; \; \; - \dfrac{1}{\Delta t} \int_0^{\Delta t} \int_{\OQ} \psi^0(r,\eta,\y) \psit(r,\eta,\y,t-\Delta t)\ d\q\dd \y d t.
\end{aligned}
\label{dif-fin}
\end{equation}
Smoothness of $\psit$ entails
%
%
\begin{equation}
 \dfrac{1}{\Delta t} \int_0^{\Delta t} \int_{\OQ} \psi^0(r,\eta,\y) \psit(r,\eta,\y,t-\Delta t)\  d\q \dd \y d t \rightarrow \int_{\OQ} \psi^0 \psit(t=0) \ d\q \dd \y,
\label{cv-in}
\end{equation}
and
\begin{equation}
\dfrac{\psit(r,\eta,\y,t)-\psit(r,\eta,\y,t-\Delta t)}{\Delta t} \rightarrow
\deriv{t}\psit(r,\eta,\y,t)  \quad strongly\ in \; L^2(0, T; L_\alpha^2).
\label{cv-df1}
\end{equation}
We also have
\begin{equation}
\dfrac{\psit(r,\eta,\x_n(\y),t)-\psit(r,\eta,\y,t)}{\Delta t}
= \nabla_y \psit (r, \eta, \y + \theta_1 (\x_n(\y) - \y), t) \cdot \xi_N,
\end{equation}
with $ \theta_1 \in ]0, 1[$ \ and
\begin{equation}
\xi_N = \frac {\x_n(\y) - \y} {\Delta t}.
\end{equation}
Since $ \x_n(y) = \chi^n(t_{n-1}, t_n, \y)$ we have
\begin{equation}
\begin{aligned}
\xi_N & = \frac{\chi^n(t_{n-1}, t_n, \y) - \chi^n(t_n, t_n, \y)} {\Delta t},\\
 &= -
\frac {\partial \chi^n} {\partial t} (t_{n-1} + \theta_2 \Delta t, t_n, \y)\,\
& = - \uu^n(\chi^n (t_{n-1} + \theta_2 \Delta t, t_n, \y)),
\end{aligned}
\end{equation}
with \ $ \theta_2 \in ]0, 1[$.
Then
\begin{equation}
\begin{aligned}
&\dfrac{\psit(r,\eta,\x_n(\y),t)-\psit(r,\eta,\y,t)}{\Delta t}\\
&\qquad \qquad = - \nabla_y \psit (r, \eta, \y + \theta_1 (\x_n(\y) - \y), t) \cdot \uu^n(\chi^n (t_{n-1} + \theta_2 \Delta t, t_n, \y)).
\end{aligned}
\label{eq-df}
\end{equation}
On the other hand, for any $s \in [t_{n-1}, t_n]$
\begin{equation}
\begin{aligned}
 \chi^n (s; t_n, \y) - \y &  =  \chi^n (s; t_n, \y) - \chi^n (t_n; t_n, \y),\\
 & = \frac {\partial \chi^n} {\partial t} (t_{n} + \theta_3 (s - t_n), t_n, \y) (s - t_n),
\\
& = \uu^n (\chi^n(t_{n} + \theta_3 (s - t_n), t_n, \y) ) (s - t_n),
\end{aligned}
\end{equation}
with $\theta_3 \in ]0, 1[$, \ then
\begin{equation}
|\chi^n (s; t_n, \y) - \y| \leq |s - t_n|\ \|\uu\|_{L^\infty(\Omega \times ]0, T[)}.
\label{ch-sm}
\end{equation}
Then one deduces from  \eqref{eq-df} and \eqref{ch-sm}:
\begin{equation}
\dfrac{\psit(r,\eta,\x_N(\y, t),t)-\psit(r,\eta,\y,t)}{\Delta t} \rightarrow - \uu(t, \y) \cdot
\nabla_{\y} \psit(r, \eta, \y, t),
\label{cv-df2}
\end{equation}
strongly in $L^2(0, T; L_\alpha^2)$.
Next, from \eqref{dif-fin}, \eqref{cv-in}, \eqref{cv-df1} and \eqref{cv-df2} one gets
%
%
\begin{equation}
\begin{aligned}
&\int_0^T\int_{\OQ} \dfrac{\psi_N^{+}(r,\eta,\y,t) -\psi_N^{-}(r,\eta,\z_N(\y,t),t)}{\Delta t}\ \psit(r,\eta,\y,t)\  d\q\dd \y d t\\
& \rightarrow -\int_0^T\int_{\OQ} \psi \left( \deriv{t}\psit+ \uu \cdot \nabla_\y \psit \right)\  d\q\dd \y d t - \int_{\OQ} \psi^0 \psit(t=0)\ d\q d \y.
\label{res1}
\end{aligned}
\end{equation}
Now, from the strong convergences
\begin{align}
&\nabla_{\y} \uu_N^+ \rightarrow \nabla_{\y} \uu,\\
&g(\nabla_{\y} \uu_N^+,\uu_N^+, \eta)  \rightarrow g(\gu,\uu, \eta),
\end{align}
and the fact that
\begin{equation}
\phi_N^- \rightarrow  \phi,
\end{equation}
one easily calculates the limit in \eqref{e-psn} and gets \eqref{vfpsi}.
Moreover,
\begin{equation}
\begin{split}
\int_0^T \int_{\Omega} & \dfrac{\phi_N^{+}(\y,t) -\phi_N^{-}(\y,t)}{\Delta t}\ \phit(\y,t)\ \dd \y d t\\
& \rightarrow  - \int_{\Omega} \psi^0 \psit(t=0)\ d \y - \int_0^T \int_{\Omega} \phi  \deriv{t}\phit\ \dd \y d t.
\end{split}
\end{equation}
Calculating the limit in \eqref{e-phn} easily leads to  \eqref{vfphi}.
\section{Conclusions}
Understanding polymer dynamics under different experimental conditions is of importance for the laboratory biologists. In this work we studied the influence of an external velocity field on the polymer-fibrils fragmentation (scission) and lengthening process. To the best of our knowledge this type of study has never been taken into account in the mathematical modelling of this problem. And even if our approach is at its early stage of development, we managed to obtain a rather good generalization of the existing models using more realistic assumptions
when adapted to the prion study.

 
In this work, we generalized the corresponding Fokker-Planck-Smoluchowski partial differential equation for rigid rods in order to account for the fragmentation/lengthening process adapted for prion proliferation. Moreover,
we have introduced a set of two equations on monomers and polymers with a known flow. We prove existence and positivity of weak solutions to the system with assumptions on the rates and distribution kernel. The proof is based on variational formulation, a semi-discretization in time, and we obtain estimations which allow us to pass to the limit. To achieve this, we introduced a suitable functional framework (see section \ref{functionalframework}).

The matter of existence of solutions to the full system (\textit{i.e.} considering the time dependence of monomers
together with the Navier-Stokes equations given in section \ref{themodel}) will be adressed in a future work.
%


\section*{Acknowledgments}
The authors gratefully acknowledge Dr. Jean-Pierre Liautard, directeur de recherche \`{a} l'\textsc{INSERM},  Universit\'e de Montpellier 2, France, for providing the image in figure \ref{fig1} and for useful talks on biology of prions. \\
This work was supported by ANR grant MADCOW no. 08-JCJC-0135-CSD5.
%
\section*{Appendix}
Let $M\in \mathcal{M}_3(\RR)$, $\eta\in\Sd$, we shall compute in spherical coordinates according to the base $(e_\theta, e_\varphi, e_r)$
\[\nabla_\eta \cdot \mathrm{P}_{\eta^\perp} M \eta = \nabla_\eta \cdot M\eta - \nabla_\eta \cdot ( M\eta\cdot \eta) \eta. \]
Note that in spherical coordinates, $\eta=e_r$ and for $F$ a vector value function,
\[\nabla_\eta\cdot F = \partial_\theta F_\theta + \dfrac{\cos\theta}{\sin\theta} F_\theta + \dfrac{1}{\sin\theta}\partial_\varphi F_\varphi + 2 F_r,\]
with $F_k = F \cdot e_k$, for $k = \theta, \varphi, r$. According to the derivative of the vector of the base, see Appendix II \cite{Otto2008} and the fact that
\[\partial_k M e_r\cdot e_j = M\partial_k e_r\cdot e_j + M e_r\cdot \partial_k e_j,\]
assumed that  $F = M e_r$, then
\[\nabla_\eta\cdot M e_r  = M e_\theta\cdot e_\theta + M e_\varphi\cdot e_\varphi.\]
Next, take $F = (Me_r\cdot e_r)e_r$, it is clear that
\[F_\theta =  (Me_r\cdot e_r)(e_r\cdot e_\theta) = 0,\;\;\text{and}\;\;F_\varphi = (Me_r\cdot e_r)(e_r\cdot e_\varphi) = 0,\]
thus
\[\nabla_\eta\cdot (Me_r\cdot e_r)e_r =  2 Me_r \cdot e_r.\]
Finally,
\[\nabla_\eta \cdot \mathrm{P}_{\eta^\perp} M \eta = M e_\theta\cdot e_\theta + M e_\varphi\cdot e_\varphi - 2 M e_r \cdot e_r.\]
%
%

\bibliographystyle{abbrv}
\bibliography{dcdsb}

\begin{thebibliography}{10}

\bibitem{Bird1987}
R.~Bird, R.~Armstrong, and O.~Hassager.
\newblock Dynamics of polymeric liquids, vol. 2: Kinetic theory.
\newblock {\em A Wiley-Interscience Publication, John Wiley \& Sons}, 1987.

\bibitem{Calvez2009}
V.~Calvez, N.~Lenuzza, D.~Oelz, J.~Deslys, P.~Laurent, F.~Mouthon, and
  B.~Perthame.
\newblock Size distribution dependence of prion aggregates infectivity.
\newblock {\em Mathematical Biosciences}, 217(1):88--99, 2009.

\bibitem{Caughey2009}
B.~Caughey, G.~Baron, B.~Chesebro, and M.~Jeffrey.
\newblock Getting a grip on prions: oligomers, amyloids and pathological
  membrane interactions.
\newblock {\em Annual review of biochemistry}, 78:177, 2009.

\bibitem{Doumic2009}
M.~Doumic, T.~Goudon, and T.~Lepoutre.
\newblock Scaling limit of a discrete prion dynamics model.
\newblock {\em Communications in Mathematical Sciences}, 7(4):839--865, 2009.

\bibitem{Engler2006}
H.~Engler, J.~Pr{\"u}ss, and G.~Webb.
\newblock Analysis of a model for the dynamics of prions ii.
\newblock {\em Journal of mathematical analysis and applications},
  324(1):98--117, 2006.

\bibitem{Greer2006}
M.~Greer, L.~Pujo-Menjouet, and G.~Webb.
\newblock A mathematical analysis of the dynamics of prion proliferation.
\newblock {\em Journal of theoretical biology}, 242(3):598--606, 2006.

\bibitem{Greer2007}
M.~Greer, P.~Van~den Driessche, L.~Wang, and G.~Webb.
\newblock Effects of general incidence and polymer joining on nucleated
  polymerization in a model of prion proliferation.
\newblock {\em SIAM Journal on Applied Mathematics}, 68:154, 2007.

\bibitem{Huilgol1997}
R.~Huilgol and N.~Phan-Thien.
\newblock {\em Fluid mechanics of viscoelasticity}.
\newblock Elsevier, 1997.

\bibitem{Kirkwood1968}
J.~G. Kirkwood.
\newblock {\em Macromolecules}.
\newblock Gordon and Breach, 1968.

\bibitem{Lansbury1995}
P.~Lansbury et~al.
\newblock The chemistry of scrapie infection: implications of the ``ice 9''
  metaphor.
\newblock {\em Chemistry \& biology}, 2(1):1--5, 1995.

\bibitem{Laurencot2007}
P.~Lauren{\c{c}}ot and C.~Walker.
\newblock Well-posedness for a model of prion proliferation dynamics.
\newblock {\em Journal of Evolution Equations}, 7(2):241--264, 2007.

\bibitem{Masel1999}
J.~Masel, V.~Jansen, and M.~Nowak.
\newblock Quantifying the kinetic parameters of prion replication.
\newblock {\em Biophysical chemistry}, 77(2-3):139--152, 1999.

\bibitem{Otto2008}
F.~Otto and A.~Tzavaras.
\newblock Continuity of velocity gradients in suspensions of rod--like
  molecules.
\newblock {\em Communications in Mathematical Physics}, 277(3):729--758, 2008.

\bibitem{Prusiner1998}
S.~Prusiner.
\newblock Prions.
\newblock {\em Proceedings of the National Academy of Sciences},
  95(23):13363--13383, 1998.

\bibitem{Pruss2006}
J.~Pr{\"u}ss, L.~Pujo-Menjouet, G.~Webb, and R.~Zacher.
\newblock Analysis of a model for the dynamics of prions.
\newblock {\em Discrete Contin. Dyn. Syst. Ser. B}, 6(1):225--235, 2006.

\bibitem{Scheibel2001}
T.~Scheibel, A.~Kowal, J.~Bloom, and S.~Lindquist.
\newblock Bidirectional amyloid fiber growth for a yeast prion determinant.
\newblock {\em Current Biology}, 11(5):366--369, 2001.

\bibitem{Simonett2006}
G.~Simonett and C.~Walker.
\newblock On the solvability of a mathematical model for prion proliferation.
\newblock {\em Journal of mathematical analysis and applications},
  324(1):580--603, 2006.

\bibitem{Walker2007}
C.~Walker.
\newblock Prion proliferation with unbounded polymerization rates.
\newblock In E.~J. of~Differential~Equations, editor, {\em Proceedings of the
  Sixth Mississippi State Conference on Differential Equations and
  Computational Simulations}, Conference 15, pages 387--397, 2007.

\bibitem{Zomosa-Signoret2007}
V.~Zomosa-Signoret, J.~Arnaud, P.~Fontes, M.~Alvarez-Martinez, and J.~Liautard.
\newblock Physiological role of the cellular prion protein.
\newblock {\em Veterinary research}, 39(4):9, 2007.

\end{thebibliography}

\end{document}